\documentclass[12pt, leqno,twoside]{article}
\usepackage{amssymb}
\usepackage{amsmath}
\usepackage{amsthm}
\usepackage{graphicx}
\usepackage{cite}

\usepackage{amsmath}
\usepackage{amsfonts}
\usepackage{amssymb}
\usepackage{amsmath,bbm,amssymb,amsxtra}
\usepackage{mathrsfs}
\usepackage{enumerate}
\usepackage{caption}
\allowdisplaybreaks
\usepackage{epsfig}
\usepackage{ae}

  \def\Xint#1{\mathchoice
    {\XXint\displaystyle\textstyle{#1}}%
    {\XXint\textstyle\scriptstyle{#1}}%
    {\XXint\scriptstyle\scriptscriptstyle{#1}}%
    {\XXint\scriptscriptstyle\scriptscriptstyle{#1}}%
    \!\int}

    \def\XXint#1#2#3{{\setbox0=\hbox{$#1{#2#3}{\int}$ }
    \vcenter{\hbox{$#2#3$ }}\kern-.6\wd0}}
    \def\dashint{\Xint-}

\newcommand{\Besov}{{B^{\theta}_{p, p}(\partial X)}}
\newcommand{\dBesov}{{\dot{ B}^{\theta}_{p}(\partial X)}}

\newcommand{\Newton}{{N^{1, p}(X, \mu_\lambda)}}

\newcommand{\lbesovo}{{\mathcal B^{0, \lambda}_\alpha(\partial X)}}
\newcommand{\dlbesovo}{{\dot{\mathcal B}^{0, \lambda}_\alpha(\partial X)}}
\newcommand{\dlbesovoe}{{\dot{\mathcal B}^{0, \lambda}_\alpha}}
\newcommand{\lbesov}{{\mathcal B^{\theta, \lambda}_{p}(\partial X)}}

\newcommand{\dlbesov}{{\dot{\mathcal B}^{\theta, \lambda}_{p}(\partial X)}}
\newcommand{\dlbesove}{{\dot{\mathcal B}^{\theta, \lambda}_{p}}}
\newcommand{\lpb}{{L^p(\partial X)}}
\newcommand{\bx}{{{\partial X}}}
\newcommand{\mul}{{{\mu_\lambda}}}

\newcommand{\Tr}{{\rm Tr}\,}
\newcommand{\Ex}{{\rm Ex}\,}
\newcommand{\Id}{{\rm Id}\,}

\newcommand{\lbd}{{\lambda}}

\newcommand{\Lip}{{\rm Lip}\,}
\newcommand{\LIP}{{\rm LIP}\,}
\newcommand{\N}{{\mathbb N}}

\newcommand{\dyadic}{\mathscr{Q}}
\newcommand{\diam}{\text{\rm\,diam}}

\newcommand{\real}{{\mathbb R}}
\newcommand{\dist}{\text{\rm dist}}

\newcommand{\rarrow}{\rightarrow}
\newtheorem{thm}{Theorem}[section]
\newtheorem{lem}[thm]{Lemma}
\newtheorem{prop}[thm]{Proposition}
\newtheorem{rem}[thm]{Remark}
\newtheorem{cor}[thm]{Corollary}
\newtheorem{defn}[thm]{Definition}
\newtheorem{example}[thm]{Example}
\numberwithin{equation}{section}

\begin{document}
\title{\Large\bf  Dyadic norm Besov-type spaces as trace spaces on regular trees
\footnotetext{\hspace{-0.35cm}
$2010$ {\it Mathematics Subject classfication}: 46E35, 30L05
\endgraf{{\it Key words and phases}: Besov-type space, regular tree, trace space, dyadic norm}
\endgraf{Authors have been supported by the Academy of Finland via  Centre of Excellence in Analysis and Dynamics Research (project No. 307333).}
}}
\author{Pekka Koskela,  Zhuang Wang}
\date{ }
\maketitle
\begin{abstract}
In this paper, we study function spaces defined via dyadic energies on the boundaries of regular trees. 
We show that correct choices of dyadic energies result in  Besov-type spaces that are trace spaces of (weighted) first order Sobolev spaces.
\end{abstract}

\section{Introduction}

Over the past two decades, analysis on general metric measure spaces has attracted a lot of attention, e.g., \cite{BB11,BBS03,H03,HP,H01,HK98,HKST}. Especially, the case of a regular tree and its  Cantor-type boundary has been studied in \cite{BBGS}. Furthermore, Sobolev spaces, Besov spaces and Triebel-Lizorkin spaces on metric measure spaces have been studied in \cite{SS,BS18,So} via hyperbolic fillings. A related approach was used in \cite{KTW}, where the trace results of Sobolev spaces and of related fractional smoothness function spaces were recovered by using a dyadic norm and the Whitney extension operator. 


 Dyadic energy has also been  used to study the regularity and modulus of continuity of space-filling curves. One of the motivations for this paper is the approach in \cite{KKZ}. Given a continuous $g: S^1\rightarrow \real^n$, consider the dyadic energy
\begin{equation}\label{dyadic-energy}
\mathcal E(g; p, \lambda):=\sum_{i=1}^{+\infty}i^\lambda\sum_{j=1}^{2^i}|g_{I_{i, j}}-g_{\widehat I_{i,j}}|^p.
\end{equation}
 Here, $\{I_{i,j}: i\in \mathbb N,  j=1, \cdots, 2^i\}$ is a dyadic decomposition of $S^1$ such that for every fixed $i\in \mathbb N$, $\{I_{i,j}: j=1, \cdots, 2^i\}$ is a family of arcs of length $2\pi/2^i$ with $\bigcup_{j}I_{i,j} = S^1$. The next generation is constructed in such a way that for each $j\in\{1, \cdots, 2^{i+1}\}$, there exists a unique number $k\in \{1, \cdots, 2^i\}$, satisfying $I_{i+1,j}\subset I_{i,k}$. We denote this parent of $I_{i+1,j}$ by $\widehat I_{i+1,j}$ and set $\widehat I_{1,j}=S^1$ for $j=1,2$. By $g_A$, $A\subset S^1$, we denote the mean value $g_A=\dashint_A g \, d\mathcal H^1=\frac{1}{\mathcal H^1(A)}\int_A g\, d\mathcal H^1$. One could expect to be able to use the energy \eqref{dyadic-energy} to characterize the trace spaces of some Sobolev spaces (with suitable weights) on the unit disk. On the contrary, the results in  \cite{KTW} suggest that the trace spaces of Sobolev spaces (with suitable weights) on the unit disk should be characterized by the energy
 \begin{equation}\label{dyadic-energy2}
\mathbb E(g; p, \lambda):=\sum_{i=1}^{+\infty}i^\lambda\sum_{j=1}^{2^i}|g_{I_{i, j}}-g_{ I_{i,j-1}}|^p,
\end{equation}
where $I_{i, 0}=I_{i,2^i}$, and the example $g(x)=\chi_{I_{1,1}}$ shows that $\mathcal E(g; p, \lambda)$ is not comparable to $\mathbb E(g; p, \lambda)$.

Notice that the energies \eqref{dyadic-energy} and \eqref{dyadic-energy2} can be viewed as dyadic energies on the boundary of a binary tree ($2$-regular tree). More precisely, for a $2$-regular tree $X$ in Section \ref{s-regular} with $\epsilon=\log 2$ in the metric \eqref{metric}, the measure $\nu$ on the boundary $\bx$ is the Hausdorff $1$-measure by Proposition \ref{Ahlfor-boundary}. Furthermore, there is a one-to-one map $h$ from the dyadic decomposition of $S^1$ to the dyadic decomposition of $\bx$ defined in Section \ref{s-besov}, which preserves the parent relation, i.e., $h(\widehat I)=\widehat {h(I)}$ for all dyadic intervals $I$ of $S^1$. Since every point in $S^1$ is the limit of a sequence of dyadic intervals, we can define a map $\tilde h$  from $S^1$ to $\bx$  by mapping any point $x=\bigcap_{k\in \mathbb N} I_k$ in $S^1$ to the limit of $\{h(I_k)\}_{k\in \mathbb N}$ (if the limit is not unique for different choices of sequence $\{I_k\}$ for $x$, then just pick one of them). It follows from the definition of $\bx$ that the map $\tilde h$ is an injective map. Since the measure $\nu$ is the Hausdorff $1$-measure and $\bx\setminus \tilde h(S^1)$ is a set of countably many points, it follows from the definition of Hausdorff measure that $\nu(\bx\setminus \tilde h(S^1))=0$.  Since $\diam (I) \approx \diam (h(I))$ for any dyadic interval $I$ of $S^1$ and  we can use dyadic intervals to cover a given set  in the definition of a Hausdorff measure,  there is a constant $C\geq 1$ such that
$$\frac 1C\mathcal H^1(A)\leq\nu(\tilde h(A))\leq C \mathcal H^1(A)$$
for any measurable set $A\subset S^1$. Then one could expect to be able to use an energy similar to \eqref{dyadic-energy2}, the $\dot{\mathbb B}_p^{1/p, \lambda}$-energy given by
\begin{equation}\label{dyadic-energy3}
\|g\|^p_{\dot{\mathbb B}_p^{1/p, \lambda}}:= \sum_{i=1}^{\infty} i^\lambda \sum_{j=1}^{2^i}\left|g_{h(I_{j, i})}-g_{ h(I_{j, i-1})}\right|^p,
\end{equation} 
to characterize the trace spaces of suitable Sobolev spaces of the $2$-regular tree. This turns out to hold in the sense that any function in $L^p(\bx)$ with finite $\dot{\mathbb B}_p^{1/p, \lambda}$-energy can be extended to a function  in a certain Sobolev class.

However, there exists a Sobolev function whose trace function has infinite $\dot{\mathbb B}_p^{1/p, \lambda}$-energy. More precisely, let $0$ be the root of the tree $X$ and let $x_1, x_2$ be the two children of $0$. We define a function  $u$ on $X$  by setting  $u(x)=0$ if the geodesic from $0$ to $x$ passes through $x_1$, $u(x)=1$ if the geodesic from $0$ to $x$ passes through $x_2$ and define $u$ to be linear on the geodesic $[x_1, x_2]=[0, x_1]\cup[0, x_2]$. Then $u$ is a Sobolev function on $X$ with the trace function $g=\chi_{h(I_{1,1})}$ whose $\dot{\mathbb B}_p^{1/p, \lambda}$-energy is not finite for any $\lambda\geq -1$, since the energy \eqref{dyadic-energy2}  of the function $\chi_{I_{1,1}}$ is not finite for any $\lambda\geq -1$. But the energy \eqref{dyadic-energy} of  the function $\chi_{I_{1,1}}$ is finite. Hence, rather than studying the energy \eqref{dyadic-energy3}, we shall work with an energy similar to  \eqref{dyadic-energy}. We define the dyadic $\dot{\mathcal B}_p^{1/p, \lambda}$ energy by setting
$$\|g\|^p_{\dot{\mathcal B}_p^{1/p, \lambda}}:=\sum_{i=1}^{\infty} i^\lambda \sum_{j=1}^{2^i} \left|g_{h(I_{i, j})}-g_{ h(\widehat I_{i, j})}\right|^p= \sum_{i=1}^{\infty} i^\lambda \sum_{I\in\dyadic_i}\left|g_{I}-g_{ \widehat I}\right|^p,$$
where $\dyadic=\cup_{j\in \mathbb N} \dyadic_j$ is a dyadic decomposition on the boundary  of the $2$-regular tree in Section \ref{s-besov}. 

Instead of only considering the above dyadic energy on the boundary of a $2$-regular tree, we introduce a general dyadic energy $\dlbesove$ in Definition \ref{besov-3},  defined on the boundary of any regular tree and for any $0\leq\theta<1$.
It is natural to ask whether the Besov-type space $\lbesov$ in Definition \ref{besov-3} defined via  the $\dlbesove$-energy is a trace space of a suitable Sobolev space defined on the regular tree.  We refer to \cite{Ar,FJS,Ga,HS10,J00,Pe,T83,T01,Tyu1,JoWa80,KTW,Tyu2} for trace results on Euclidean spaces and to \cite{BBGS,SS,Ma} for trace results on metric measure spaces.

  In \cite{BBGS},  the trace spaces of the Newtonian spaces $N^{1,p}(X)$ on regular trees were shown to be Besov spaces defined via  double integrals. Our first result is the following generalization of this theorem.


\begin{thm}\label{th2} Let $X$ be a $K$-ary tree with $K\geq 2$. Fix $\beta>\log K$, $\epsilon>0$ and $\lambda\in \real$. Suppose that $p\geq 1$ and $p>(\beta-\log K)/\epsilon$. Then the Besov-type space $\lbesov$ is the trace space of $N^{1, p}(X, \mu_\lbd)$ whenever $\theta=1-(\beta-\log K)/\epsilon p$.
\end{thm}

The measure $\mul$ above is defined in \eqref{measure-lbd} by
$$d\mu_\lambda(x)=e^{-\beta|x|}(|x|+C)^\lbd\, d|x|,$$
and the space $N^{1, p}(X, \mu_\lbd)$ is a Newtonian space defined in Section \ref{Newtonian-space}. If $\lambda=0$, then $N^{1, p}(X, \mul)=N^{1,p}(X)$ and Theorem \ref{th2} recovers the trace results from \cite{BBGS} for the Newtonian spaces $N^{1, p}(X)$.
 Here and throughout this paper, for given Banach spaces $\mathbb X(\bx)$ and $\mathbb Y(X)$, we say that the space $\mathbb X(\bx)$ is a trace space of $\mathbb Y(X)$ if and only if there is a bounded linear operator $T: \mathbb Y(X)\rightarrow \mathbb X(\bx)$ and there exists a bounded linear extension operator  $E: \mathbb X(\bx)\rightarrow \mathbb Y(X)$ that acts as a right inverse of $T$, i.e., $T\circ E=\Id$ on the space $\mathbb X(\bx)$.

We required in Theorem \ref{th2} that $p>(\beta-\log K)/\epsilon>0$. The assumption that    $\beta-\log K>0$ is necessary in the sense that we need to make sure that the measure $\mu_\lambda$ on  $X$ is doubling; see Section \ref{section-measure}. The requirement that  $p>(\beta-\log K)/\epsilon$ will ensure that $\theta>0$. So it is natural to consider the case  $p=(\beta-\log K)/\epsilon\geq 1$. 
\begin{thm}\label{th3}
 Let $X$ be a $K$-ary tree with $K\geq 2$. Fix $\beta>\log K$, $\epsilon>0$ and $\lambda\in \real$. Suppose that $p=(\beta-\log K)/\epsilon\geq 1$ and $\lambda>p-1$ if $p>1$ or $\lambda\geq 0$ if $p=1$. Then there is a bounded linear trace operator $T: N^{1, p}(X, \mu_\lambda)\rightarrow L^p(\bx)$, defined via limits along geodesic rays. Here, $\lambda>p-1$ is sharp  in the sense that for any $p>1$, $\delta>0$ and $\lambda=p-1-\delta$, there exists a function $u\in \Newton$ so that $Tu(\xi)=\infty$ for every $\xi\in \bx$.

Moreover, for any $p=(\beta-\log K)/\epsilon\geq 1$, there exists a bounded nonlinear extension operator $E: L^p(\bx)\rightarrow N^{1, p}(X)$ that acts as a right inverse of the trace operator $T$ above, i.e., $T\circ E=\Id$ on $L^p(X)$.
\end{thm}


A result similar to Theorem \ref{th3} for the weighted Newtonian space 
$N^{1, p}(\Omega, \omega\, d\mu)$ with a suitable weight $\omega$ has been established in \cite{Ma} provided that $\Omega$ is a bounded domain that admits a 
$p$-Poincar\'e inequality and whose boundary $\partial \Omega$ is endowed with a $p$-co-dimensional Ahlfors regular measure. In Theorem \ref{th3}, for the 
case $p=(\beta-\log K)/\epsilon> 1$, we require that $\lambda>p-1$ to ensure 
the existence of limits along geodesic rays. In the case $p=(\beta-\log K)/\epsilon= 1$,  these limits  exist even for $\lambda=0$, and there is a nonlinear extension operator that acts as a right inverse of the trace operator,  similarly to the case of $W^{1, 1}$ in Euclidean setting; see \cite{Ga,Pe}.

However, except for the case $p=1$ and $\lambda=0$, Theorem \ref{th3} does not even tell whether the trace operator $T$ is surjective or not: $N^{1, p}(X, \mul)$ is a strict subset of $N^{1, p}(X)$ when $\lambda>0$. In the case $p=(\beta-\log K)/\epsilon= 1$ and  $\lambda>0$, the trace operator $T$ is actually not 
surjective, and we can find a Besov-type space  $\lbesovo$ (see Definition \ref{besov-4}) which is the trace space of the Newtonian space $N^{1, 1}(X, \mul)$. We stress that  $\lbesovo$ and ${\mathcal B}^{0, \lambda}_1(\bx)$ are different spaces. More precisely, ${\mathcal B}^{0, \lambda}_1(\bx)$ is a strict subspace of $\lbesovo$, see Proposition \ref{compare} and Example \ref{exam-strict}.

\begin{thm}\label{th5} Let $X$ be a $K$-ary tree with $K\geq 2$. Fix $\beta>\log K$, $\epsilon>0$ and $\lambda >0$.
Suppose that $p=1=(\beta-\log K)/\epsilon$. Then the trace space of $N^{1,1}(X, \mu_\lambda)$ is the Besov-type space $\lbesovo$.
\end{thm}

Trace results similar to Theorem \ref{th5} in the Euclidean setting can be found in \cite{Gin84,Tyu2}.
The second part of Theorem \ref{th3} asserts the existence of bounded nonlinear extension operator from $L^{p}(\bx)$ to $N^{1,p}(X)$ whenever $p=(\beta-\log K)/\epsilon\geq 1$. Nonlinearity is natural here since results due to 
Peetre \cite{Pe} (also see \cite{BG79}) indicate that, for $p=1$ and $\lambda=0$, one can not find a 
bounded linear extension operator that acts as a right inverse of the trace 
operator in Theorem \ref{th3}. On the other hand, the recent work \cite{MSS} gives the existence of a bounded linear extension operator $E$ from a certain Besov-type space to $BV$ or to $N^{1,1}$ such that $T\circ E$ is the identity 
operator on this Besov-type space, under the assumption that the domain 
satisfies the co-dimension $1$ Ahlfors-regularity. 
The extension operator in \cite{MSS} is a version of the Whitney extension operator. This motivates us to further analyze the operator $E$ from Theorem \
\ref{th2}: it is also of Whitney type.
The co-dimension $1$ Ahlfors-regularity does not hold for our regular tree $(X, \mul)$, but we are still able to establish the following result for $N^{1, p}(X, \mul)$ with $p\geq 1$ for our fixed extension operator $E$.

\begin{thm}\label{th4} Let $X$ be a $K$-ary tree with $K\geq 2$. Fix $\beta>\log K$, $\epsilon>0$ and $\lambda\in \real$. 
Suppose that $p=(\beta-\log K)/\epsilon\geq 1$ and $\lambda>p-1$ if $p>1$ or $\lambda\geq 0$ if $p=1$. Then the operator $E$ from Theorem \ref{th2} is a bounded linear extension operator from ${\mathcal B}^{0, \lambda}_p(\bx)$ to $N^{1, p}(X, \mul)$  and acts as a right inverse of $T$, i.e., $T\circ E$ is the identity operator on ${\mathcal B}^{0, \lambda}_p(\bx)$, where $T$ is the trace operator  in Theorem \ref{th3}. 

Moreover, 
the space ${\mathcal B}^{0, \lambda}_p(\bx)$ is the optimal space for which $E$ is both bounded and linear, i.e., if  $\mathbb X\subset L^1_{\rm loc}(\bx)$  is a Banach space so that the extension operator $E: \mathbb X\rarrow N^{1, p}(X, \mul)$ is bounded and linear and so that $T\circ E$ is the identity operator on $\mathbb X$, then $\mathbb X$ is a subspace of ${\mathcal B}^{0, \lambda}_p(\bx)$.
\end{thm}

The optimality of the space ${\mathcal B}^{0, \lambda}_p(\bx)$ is for the explicit extension operator $E$  in Theorem \ref{th4}.  The space 
${\mathcal B}^{0, \lambda}_p(\bx)$ may not be the optimal space unless we 
consider this particular extension operator. For example, for $p=1$ and $\lambda>0$, the optimal space is $\lbesovo$ rather than ${\mathcal B}^{0, \lambda}_1$ by Theorem \ref{th5}. This splitting happens since the two extension operators from Theorem \ref{th5} and Theorem \ref{th4} are very different: the latter one is of Whitney type while the former one relies on the same dyadic elements for several different dyadic layers. 

The paper is organized as follows. In Section \ref{s2}, we give all the preliminaries for the proofs. More precisely, we introduce regular trees in Section \ref{s-regular} and we consider the doubling condition on a regular tree $X$ and the Hausdorff dimension of its boundary $\bx$.  We introduce the Newtonian spaces on $X$ and the Besov-type spaces on $\bx$ in Section \ref{Newtonian-space} and Section \ref{s-besov}, respectively. In Section \ref{proofs}, we give the proofs of all the above mentioned theorems, one by one. 

In what follows, the letter $C$ denotes a constant that may change at different occurrences. The notation $A\approx B$ means that there is a constant $C$ such that $1/C\cdot A\leq B\leq C\cdot A$. The notation $A\lesssim B$ ($A\gtrsim B$) means that there is a constant $C$ such that  $A\leq  C\cdot B$ ($A\geq C\cdot B$).

\section{Preliminaries}\label{s2}
\subsection{Regular trees and their boundaries}\label{s-regular}
A {\it graph} $G$ is a pair $(V, E)$, where $V$ is a set of vertices and $E$ is a set of edges.   We call a pair of vertices $x, y\in V$  neighbors if $x$ is connected to $y$ by an edge. The degree of a vertex is the number of its neighbors. The graph structure gives rise to a natural connectivity structure. A {\it tree} is a connected graph without cycles. A graph (or tree) is made into a metric graph by considering each edge as a geodesic of length one.

We call a tree $X$ a {\it rooted tree} if it has a distinguished vertex called the {\it root}, which we will denote by $0$. The neighbors of a vertex $x\in X$ are of two types: the neighbors that are closer to the root are called {\it parents} of $x$ and all other  neighbors  are called {\it children} of $x$. Each vertex has a unique parent, except for the root itself that has none. 

A {\it $K$-ary tree}  is a rooted tree such that each vertex has exactly $K$ children. Then all vertices except the root  of  a $K$-ary tree have degree $K+1$, and the root has degree $K$. In this paper we say that a tree is {\it regular} if it is a $K$-ary tree for some $K\geq 1$.

For $x\in X$, let $|x|$ be the distance from the root $0$ to $x$, that is, the length of the geodesic from $0$ to $x$, where the length of every edge is $1$ and we consider each edge to be an isometric copy of the unit interval.  The geodesic connecting two vertices $x, y\in V$ is denoted by $[x, y]$, and its length is denoted $|x-y|$. If $|x|<|y|$ and $x$ lies on the geodesic connecting $0$ to $y$, we write $x<y$ and call the vertex $y$ a descendant of the vertex $x$. More generally, we write $x\leq y$ if the geodesic from $0$ to $y$ passes through $x$, and in this case $|x-y|=|y|-|x|$.

Let $\epsilon>0$ be fixed. We introduce a {\it uniformizing metric} (in the sense of Bonk-Heinonen-Koskela \cite{BHK01}, see also \cite {BBGS} ) on $X$ by setting
\begin{equation}\label{metric}
d_X(x, y)=\int_{[x, y]} e^{-\epsilon|z|}\, d\,|z|.
\end{equation}
Here $d\,|z|$ is the measure which gives each edge Lebesgue measure $1$, as we consider each edge to be an isometric copy of the unit interval and the vertices are the end points of this interval. In this metric, $\diam X=2/\epsilon$ if $X$ is a $K$-ary tree with $K\geq 2$.  

Next we construct the boundary of the regular $K$-ary tree by following the arguments in \cite[Section 5] {BBGS}. We define the boundary of a tree $X$, denoted $\bx$, by completing $X$ with respect to the metric $d_X$. An equivalent construction of $\bx$ is as follows. An element $\xi$ in $\bx$ is identified with an infinite geodesic in $X$ starting at the root $0$. Then we may denote $\xi=0x_1x_2\cdots$, where $x_i$ is a vertex in $X$ with $|x_i|=i$, and $x_{i+1}$ is a child of $x_i$. Given two points $\xi, \zeta\in \bx$, there is an infinite geodesic $[\xi, \zeta]$ connecting $\xi$ and $\zeta$. Then the distance of $\xi$ and $\zeta$ is the length (with respect to the metric $d_X$) of the infinite geodesic $[\xi, \zeta]$. More precisely, if $\xi=0x_1x_2\cdots$ and $\zeta=0y_1y_2\cdots$, let $k$ be an integer with $x_k=y_k$ and $x_{k+1}\not=y_{k+1}$. Then by \eqref{metric}
$$d_X(\xi, \zeta)=2\int_{k}^{+\infty} e^{-\epsilon t}\, d t=\frac 2\epsilon e^{-\epsilon k}.$$
The restriction of $d_X$ to $\bx$ is called the {\it visual metric} on $\bx$ in Bridson-Haefliger \cite{BH99}.

The metric $d_X$ is thus defined on $\bar X$. To avoid confusion, points in $X$ are denoted by Latin letters such as $x, y$ and $z$, while for points in $\bx$ we use Greek letters such as $\xi, \zeta$ and $\omega$. Moreover, balls in $X$ will be denoted $B(x, r)$, while $B(\xi, r)$ stands for a ball in $\bx$.

Throughout the paper we assume that $1\leq p<+\infty$ and that $X$ is a $K$-ary tree with $K\geq 2$ and metric $d_X$ defined as in \eqref{metric}.
 
\subsection{Doubling condition on  $X$ and Hausdorff dimension of $\bx$}\label{section-measure}
The first aim of this section is to show that the weighted measure 
\begin{equation}\label{measure-lbd}
d\mu_\lambda(x)=e^{-\beta|x|}(|x|+C)^\lbd\, d|x|
\end{equation}
is doubling on $X$, where $\beta>\log K$, $\lambda\in \mathbb R$ and $C\geq\max\{2|\lambda|/(\beta-\log K), 2(\log 4)/\epsilon\}$ are fixed from now on. Here the lower bound of the constant $C$ will make the estimates below simpler. If $\lambda=0$, then 
$$d\mu_0(x)= e^{-\beta|x|}\, d|x|=d\mu(x),$$
which coincides with the measure used in \cite{BBGS}. If $\beta\leq \log K$, then $\mul(X)=\infty$ for the regular $K$-ary tree $X$ by \eqref{beta} below. Hence $X$ would not be doubling as $X$ is bounded.

Next we estimate the measures of balls in $X$ and show that our measure is doubling. Let 
$$B(x, r)=\{y\in X: d_X(x, y)<r\}$$
denote an open ball in $X$ with respect to the metric $d_X$. Also let
$$F(x, r)=\{y\in X: y\geq x\ \text{and}\ \ d_X(x, y)<r\}$$
denote the downward directed ``half ball''.

The following algebraic lemma and the relation between a ball and a ``half ball'' come from \cite[Lemma 3.1 and 3.2]{BBGS}.
\begin{lem}\label{algebra}
Let $\sigma>0$ and $t\in [0, 1]$. Then
$$\min\{1, \sigma\} t\leq 1-(1-t)^\sigma\leq \max\{1, \sigma\} t.$$
\end{lem}
\begin{lem}\label{relation-measure}
For every $x\in X$ and $r>0$ we have
$$F(x, r)\subset B(x, r)\subset F(z, 2r),$$
where $z\leq x$ and 
\begin{equation}\label{relationeq}
|z|=\max\left\{|x|-\frac 1\epsilon \log(1+\epsilon r e^{\epsilon |x|}), 0\right\}.
\end{equation}
\end{lem}

We begin to estimate the measure of the ball $B(x, r)$ and of the half ball $F(z, r)$.
\begin{lem}\label{lF}
If $0<r\leq e^{-\epsilon|z|}/\epsilon$, then
$$\mu_\lambda(F(z, r))\approx e^{(\epsilon-\beta)|z|}r(|z|+C)^\lambda.$$
 \end{lem} 
 \begin{proof}
 
Let $\rho>0$ be such that 
$$\int_{|z|}^{|z|+\rho} e^{-\epsilon t}\, dt=\frac{1}{\epsilon}e^{-\epsilon|z|}(1-e^{-\epsilon \rho})=r.$$
Note that for each $|z|\leq t\leq |z|+\rho$, the number of points $y\in F(z, r)$ with $|y|=t$ is approximately $K^{t-|z|}$. Hence
\begin{align}
\mul(F(z, r))\approx \int_{|z|}^{|z|+\rho} K^{t-|z|}e^{-\beta t}(t+C)^\lambda\, dt=K^{-|z|}\int_{|z|}^{|z|+\rho} e^{(\log K-\beta)t}(t+C)^\lambda\, dt.\label{beta}
\end{align}
Since
$$\left(\frac{1}{\log K-\beta} e^{(\log K-\beta)t}(t+C)^\lambda\right)'=e^{(\log K-\beta)t}(t+C)^\lambda\left(1+ \frac{\lambda}{(t+C)(\log K-\beta)}\right),$$
then for $C\geq 2|\lambda|/(\beta-\log K)$, we have
$$ \left|\frac{\lambda}{(t+C)(\log K-\beta)}\right| \leq \frac 12 \ \ \ \forall \ t>0.$$
Hence we obtain that
\begin{equation}\label{est-Fz}
\mul(F(z, r))\approx \frac{K^{-|z|}}{\beta-\log K} e^{(\log K-\beta)|z|}(|z|+C)^\lambda\left(1-e^{(\log K-\beta)\rho}\left(\frac{|z|+\rho+C}{|z|+C}\right)^\lambda\right).
\end{equation}
It is easy to check that for any $\rho>0$ and $z\in X$, we have that
$$1\leq\frac{|z|+\rho+C}{|z|+C}\leq\frac{\rho+C}{C}\leq e^{\rho/C}.$$
Therefore,
$$e^{-\frac{|\lambda|}{C}\rho}\leq \left(\frac{|z|+\rho+C}{|z|+C}\right)^\lambda\leq e^{\frac{|\lambda|}{C}\rho}\ \ \ \forall \ z\in X, \rho>0.$$
Since $C\geq 2|\lambda|/(\beta-\log K)$, we obtain that
\begin{equation}\label{C-condition}
e^{\frac 12(\log K-\beta) \rho} \leq \left(\frac{|z|+\rho+C}{|z|+C}\right)^\lambda \leq e^{-\frac 12(\log K-\beta) \rho}\ \ \ \forall \ z\in X,  \rho >0.
\end{equation}
Then for any $z\in X$ and $\rho>0$,
$$e^{(\log K-\beta)\rho}\left(\frac{|z|+\rho+C}{|z|+C}\right)^\lambda\approx e^{c(\log K-\beta)\rho}, \ \text {for some }\   \frac12\leq c\leq \frac32.$$
Hence we obtain that
\begin{align*}
\mul(F(z, r))&\approx \frac{K^{-|z|}}{\beta-\log K} e^{(\log K-\beta)|z|}(|z|+C)^\lambda \left(1-e^{c(\log K-\beta)\rho}\right)\\
&= \frac{e^{-\beta|z|}}{\beta-\log K}(|z|+C)^\lambda \left(1-(1-\epsilon r e^{\epsilon|z|})^{c (\beta-\log K)/\epsilon}\right)
\end{align*}
for some $c\in [1/2, 3/2]$. Lemma \ref{algebra} with $t=\epsilon r e^{\epsilon|z|}$ implies that
$$\mul(F(z, r))\approx e^{-\beta |z|}(|z|+C)^\lambda\epsilon r e^{\epsilon|z|}\approx e^{(\epsilon-\beta)|z|}r(|z|+C)^\lambda.$$
\end{proof}

\begin{cor}\label{r-small}
If $0<r\leq e^{-\epsilon|x|}/\epsilon$, then
$$\mu_\lambda(B(x, r))\approx e^{(\epsilon-\beta)|x|}r(|x|+C)^\lambda\approx e^{(\epsilon-\beta)|x|}r(|z|+C)^\lambda.$$
 \end{cor} 
\begin{proof}
For any $x\in X$ and $0<r\leq e^{-\epsilon|x|}/\epsilon$, let $z$ be as in Lemma \ref{relation-measure}. If $z=0$, then $B(x, r)\subset F(0, r+\rho)$, where 
$$\rho=d_X(0, x)= \frac{1}{\epsilon}(1-e^{-\epsilon|x|})\leq r$$
and $r+\rho \leq 1/\epsilon = e^{-\epsilon |z|}/\epsilon$. For $z>0$ we have
$$2r\leq \frac{e^{-\epsilon|x|}(1+\epsilon r e^{\epsilon |x|})}{\epsilon}=\frac{e^{-\epsilon |z|}}{\epsilon}.$$
Moreover, in both cases, since $r<e^{-\epsilon|x|}/\epsilon$, by Lemma \ref{relation-measure}, we have
$$|z|\leq |x|\leq |z|+\frac 1\epsilon \log(1+\epsilon r e^{\epsilon |x|})\leq |z|+\frac 1\epsilon \log 2,$$
which implies 
\begin{equation}\label{compareeq}
\left(\frac{|x|+C}{|z|+C}\right)^\lambda\approx 1.
\end{equation}
Combing \eqref{compareeq} with the fact that in both cases $1\leq e^{|x|-|z|}\leq (1+\epsilon r e^{\epsilon |x|})^{1/\epsilon}\approx 1$, the result follows by applying Lemma \ref{lF} to $F(x, r)$ and $F(z, 2r)$ (or $F(0, r+\rho)$ for $z=0$).
\end{proof}

\begin{lem}\label{geodesic}
Let $z\in X$ and $x\in \overline X$ with $z\leq x$. Then
$$\mul([z, x])\approx \mul(F(z, d_X(z, x))).$$
where $[z, x]$ denotes the geodesic in the tree $X$ joining $x$ and $z$.
\end{lem}
\begin{proof} Since $[z, x]$ is a subset of $F(z, d_X(z, x))$ by definition, we have $\mul([z, x])\leq \mul(F(z, d_X(z, x)))$. Hence it remains to show that
$$\mul([z, x])\gtrsim \mul(F(z, d_X(z, x))).$$
For any $z\in X$ and $x\in \overline X$ with $z\leq x$, we have that
$$\mul([z, x])=\int_{|z|}^{|x|} e^{-\beta t}(t+C)^\lambda\, dt,$$
where $|x|=\infty$ if $x\in \partial X$. Then
by using an  argument similar to the estimate in Lemma \ref{lF}, since $C\geq2|\lambda|/(\beta-\log K) \geq 2|\lambda|/\beta$, we have that
$$\left|\frac{\lambda}{(t+C)\beta}\right|\leq \frac12\ \ \ \forall \ t\geq 0,$$
which implies that for any $t\geq 0$,
$$ \left(-\frac 1\beta e^{-\beta t}(t+C)^\lbd\right)'=e^{-\beta t}(t+C)^\lambda\left(1-\frac{\lambda}{\beta(t+C)}\right) \approx e^{-\beta t}(t+C)^\lambda.$$
Hence we obtain that
\begin{equation}\label{est-geo}
\int_{|z|}^{|x|} e^{-\beta t}(t+C)^\lambda\, dt\approx \frac{e^{-\beta|z|}}{\beta}(|z|+C)^\lambda\left(1-e^{-\beta(|x|-|z|)}\left(\frac{|x|+C}{|z|+C}\right)^\lambda\right).
\end{equation}
Comparing the estimate \eqref{est-geo} with the estimate \eqref{est-Fz}, since $\rho=|x|-|z|$, $e^{\rho \log K}\geq 1$ and $K^{-|z|}e^{(\log K-\beta)|z|}=e^{-\beta|z|}$, we have that
$$
\int_{|z|}^{|x|} e^{-\beta t}(t+C)^\lambda\, dt\gtrsim \mul(F(z, r))\ \ {\rm with}\ \ r=d_X(z, x),
$$
which induces that
\begin{equation*}
\mul([z, x])\gtrsim \mul(F(z, r))=\mul(F(z, d_X(z, x))).
\end{equation*}

\end{proof}

\begin{cor}\label{r-middle}
Let $x\in X$ and $z$ be as in Lemma \ref{relation-measure}. Then if 
\begin{equation}\label{est-r}
\frac{e^{-\epsilon |x|}}{\epsilon}\leq r\leq \frac{1}{\epsilon}(1-e^{-\epsilon|x|}),
\end{equation}
we obtain 
$$\mul(B(x, r))\approx e^{-\beta |z|}(|z|+C)^\lambda \approx r^{\beta/\epsilon}(|z|+C)^\lambda.$$
\end{cor}
\begin{proof}
Since $r\geq e^{-\epsilon|x|}/\epsilon$, by Lemma \ref{relation-measure}, we have
$$B(x, r)\subset F(z, \infty)=F\left(z, \frac{e^{-\epsilon |z|}}{\epsilon}\right).$$
Then Lemma \ref{lF} implies
\begin{equation}\label{les}
\mul(B(x, r))\leq \mul(F(z, \infty))\lesssim e^{(\epsilon-\beta)|z|}e^{-\epsilon|z|}(|z|+C)^\lambda\approx e^{-\beta |z|}(|z|+C)^\lambda
\end{equation}

Towards the another direction,  by \eqref{relationeq} and Lemma \ref{geodesic},  we have that
$$\mul(B(x, r))\geq \mul([x, z])\gtrsim \mu(F(z, r))=e^{(\epsilon-\beta)|z|}r(|z|+C)^\lambda=e^{-\beta|z|}(|z|+C)^\lambda e^{\epsilon|z|}r.$$
Moreover, we have
\begin{align*}
e^{\epsilon|z|}r=e^{\epsilon|x|} r \cdot e^{-\epsilon(|x|-|z|)}=e^{\epsilon|x|} r(1+\epsilon r e^{\epsilon |x|})^{-1}= \frac{t}{\epsilon(1+t)}\geq \frac{1}{2\epsilon},
\end{align*}
where  $t=\epsilon r e^{\epsilon |x|}$. Here in the last inequality we used the fact that $\epsilon r e^{\epsilon |x|}\geq 1$.
Hence we obtain that
$$\mul(B(x, r))\gtrsim e^{-\beta|z|}(|z|+C)^\lambda.$$
Combing the above inequality with \eqref{les}, we finish the proof of 
$$\mul(B(x, r))\approx e^{-\beta|z|}(|z|+C)^\lambda.$$

Since $\epsilon r e^{\epsilon |x|}\geq 1$, we know that 
$$\epsilon r e^{\epsilon |x|} \leq1+\epsilon r e^{\epsilon |x|}\leq 2\epsilon r e^{\epsilon |x|}.$$
It then follows from \eqref{relationeq} that
$$e^{-\beta|z|}=e^{-\beta|x|}(1+\epsilon r e^{\epsilon |x|})^{\beta/\epsilon}\approx  r^{\beta/\epsilon}.$$
Hence we obtain that
$$e^{-\beta|z|}(|z|+C)^\lambda\approx r^{\beta/\epsilon}(|z|+C)^\lambda,$$
which finishes the proof.
\end{proof}

\begin{lem}\label{r-big}
Let $x\in X$ and $(1-e^{-\epsilon |x|})/\epsilon\leq r\leq 2\diam X$. Then 
$$\mul(B(x, r))\approx r.$$
In particular, if $x=0$, then this estimate holds for all $r\geq 0$.
\end{lem}
\begin{proof}
We have $0\in \overline{B(x, r)}$ by assumption, and hence 
$$B(x, r)\subset F(0, 2r).$$
From Lemma \ref{lF}, we have that
$$\mul(B(x, r))\leq \mul(F(0, 2r)) \lesssim r.$$
As for the lower bound, if $r<1/\epsilon$, since $0\in \overline{B(x, r)}$, letting 
$$\rho=-\frac{\log(1-\epsilon r)}{\epsilon} $$
and $x\leq x'$  with $|x'|=\rho$, then the estimate \eqref{geodesic} and Lemma \ref{lF} imply
$$\mul(B(x, r))\geq \mul([0,x'])\gtrsim \mul(F(0, r))\approx r.$$
If $1/\epsilon\leq r\leq 2\diam X=4/\epsilon$, then by Lemma \ref{geodesic}, we have that
$$\mul(B(x, r))\geq \mul(F(0, 1/\epsilon))\approx \frac{1}{\epsilon} \approx r.$$
\end{proof}

\begin{prop}\label{measure}
Let $x\in X$, $0<r\leq 2 \diam X$, $R_0=e^{-\epsilon|x|}/\epsilon$ and $z$ be as in Lemma \ref{relation-measure}. If $|x|\leq(\log 2)/\epsilon$, then 
$$\mul(B(x, r))\approx r.$$
If $|x|\geq (\log 2)/\epsilon$, then 
$$\mul(B(x, r))\approx \left\{\begin{array}{lc}
e^{(\epsilon-\beta)|x|}(|x|+C)^\lambda, &\ \ \ r\leq R_0;\\
r^{\beta/\epsilon}(|z|+C)^\lambda, &\ \ \ r\geq R_0.
\end{array}\right.$$
\end{prop}
\begin{proof}
If $|x|\leq(\log 2)/\epsilon$, then $e^{(\epsilon-\beta)|x|}\approx 1$, $(|x|+C)^\lambda\approx 1$ and the result follows from directly from Corollary \ref{r-small} and Lemma \ref{r-big}.

If $|x|\geq (\log 2)/\epsilon$ and $r\leq (1-e^{-\epsilon|x|})/\epsilon$, then the estimate follows directly from Corollary \ref{r-small} and $\ref{r-middle}$. For $r\geq (1-e^{-\epsilon|x|})/\epsilon \geq 1/2\epsilon$, since $|z|=0$, we have by Lemma \ref{r-big} that 
$$\mul(B(x, r))\approx r\approx 1\approx r^{\beta/\epsilon}(|z|+C)^\lambda.$$
\end{proof}

\begin{cor}\label{doubling-}
The measure $\mul$ is doubling, i.e., $\mul(B(x, 2r)) \lesssim \mul(B(x, r))$.
\end{cor}
\begin{proof}
In the case $|x|\leq(\log 2)/\epsilon$ and the case $|x|\geq (\log 2)/\epsilon$ with $2r\leq R_0$,  the result follows directly from Proposition \ref{measure}.

In the case $|x|\geq (\log 2)/\epsilon$ with $2r\geq R_0$, if $r\geq R_0$, then 
$$r^{\beta/\epsilon}\approx (2r)^{\beta/\epsilon};$$
if $r\leq R_0$, then
$$\frac{e^{(\epsilon-\beta)|x|}r}{(2r)^{\beta/\epsilon}}\approx \left(\frac{R_0}{r}\right)^{\beta/\epsilon-1}\approx 1.$$
Let $z_r$ and $z_{2r}$ be defined as in Lemma \ref{relation-measure} with respect to $r$ and $2r$. From Corollary \ref{r-small} and the above estimates, the doubling condition of $\mul$ follows once we prove that 
\begin{equation}\label{doubling}
\frac{|z_r|+C}{|z_{2r}|+C}\approx 1.
\end{equation}
If $r\geq (1-e^{-\epsilon|x|})/\epsilon$, then $|z_r|=|z_{2r}|=0$ give \eqref{doubling}. If $2r\geq (1-e^{-\epsilon|x|})/\epsilon \geq r$, then $r\geq (1-e^{-\epsilon|x|})/2\epsilon$ implies that 
\begin{align*}
|z_r|+C&=|x|-\frac{1}{\epsilon}\log(1+\epsilon r e^{\epsilon |x|})+C\leq |x|-\frac1\epsilon \log\big(\frac12(1+e^{\epsilon |x|})\big)+C \\
&= |x|+C+\frac{\log 2}{\epsilon}- \frac{1}{\epsilon}\log(1+e^{\epsilon |x|})\leq C+\frac{\log 2}{\epsilon}\approx C =|z_{2r}|+C,
\end{align*}
which gives \eqref{doubling}. If $2r\leq (1-e^{-\epsilon|x|})/\epsilon$, for $C\geq 2(\log 2)/\epsilon$, we obtain that 
\begin{align*}
2(|z_{2r}|+C)-(|z_r|+C) &= |x|+C+\frac{1}{\epsilon}\log(1+\epsilon r e^{\epsilon |x|})-\frac{2}{\epsilon}\log(1+2\epsilon r e^{\epsilon |x|})\\
&\geq |x|+C+\frac{1}{\epsilon}\log(1+\epsilon r e^{\epsilon |x|})-\frac{2}{\epsilon}\log(2(1+\epsilon r e^{\epsilon |x|}))\\
&=|x|+C-\frac{2\log 2}{\epsilon} -\frac{1}{\epsilon}\log(1+\epsilon r e^{\epsilon |x|})\\
&=|z_r|+C-\frac{2\log 2}{\epsilon}\geq 0,
\end{align*}
which gives that $|z_r|+C\leq 2(|z_{2r}|+C)$.  Combining with the fact that $|z_{2r}|\leq |z_r|$, \eqref{doubling} is obtained. Therefore we finish the proof of this corollary.
\end{proof}

The following result is given by \cite[Lemma 5.2]{BBGS}.
\begin{prop}\label{Ahlfor-boundary}
The boundary $\bx$ is an Ahlfors $Q$-regular space with Hausdorff dimension 
$$Q=\frac{\log K}{\epsilon}.$$
\end{prop}
Hence we have an Ahlfors $Q$-regular measure $\nu$ on $\bx$ with
$$\nu(B(\xi, r)) \approx r^Q=r^{\log K/\epsilon},$$
for any $\xi\in \bx$ and $0<r\leq \diam \bx.$

\subsection{Newtonian spaces  on $X$}\label{Newtonian-space}
Let $u\in L_{\rm loc}^1(X, \mul)$. We say that a Borel function $g: X\rarrow [0, \infty]$ is an {\it upper gradient} of $u$ if 
\begin{equation}\label{gradient}|u(z)-u(y)|\leq \int_{\gamma} g\, ds_X\end{equation}
whenever $z, y\in X$ and $\gamma$ is the geodesic from $z$ to $y$, where $d s_X$ denotes the arc length measure with respect to the metric $d_X$. In the setting of a tree any rectifiable curve with end points $z$ and $y$ contains the geodesic connecting $z$ and $y$, and therefore the upper gradient defined above is equivalent to the definition which requires that inequality \ref{gradient} holds for all rectifiable curves with end points $z$ and $y$.

The notion of upper gradients is due to Heinonen and Koskela \cite{HK98}; we refer interested readers to \cite{H03,HKST} for a more detailed discussion on upper gradients.

The {\it Newtonian space} $\Newton$, $1\leq p<\infty$, is defined as the collection of all the functions for which 
$$\|u\|_{\Newton}:= \left(\int_X |u|^p\, d\mul+\inf_g \int_X g^p\, d\mul\right)^{1/p}<\infty,$$
where the infimum is taken over all upper gradients of $u$. 

%



Throughout the paper, we use $N^{1, p}(X)$ to denote $N^{1, p}(X, \mul)$ if $\lambda=0$.

\subsection{Besov-type spaces on $\bx$ via dyadic norms}\label{s-besov}
We first recall the Besov space $\Besov$ defined in \cite{BBGS}.
\begin{defn}\label{besov-0}\rm
For $0<\theta<1$ and $p\geq 1$, The Besov space $\Besov$ consists of all functions $f\in L^p(\bx)$ for which the seminorm  $\|f\|_{\dBesov}$ defined as
$$\|f\|^p_{\dBesov}:=\int_{\bx}\int_{\bx}\frac{|f(\zeta)|-f(\xi)|^p}{d_X(\zeta, \xi)^{\theta p}\nu(B(\zeta, d_X(\zeta, \xi)))}d\nu(\xi)\, d\nu(\zeta)$$
is finite. The corresponding norm for $\Besov$ is 
$$\|f\|_{\Besov}:=\|f\|_{\lpb}+\|f\|_{\dBesov}.$$
\end{defn}

Next, we give a dyadic decomposition on the boundary $\bx$ of the $K$-ary tree $X$:
 Let $V_n=\{x_j^n:  j=1, 2, \cdots, K^n\}$ be the set of all $n$-level vertices of the tree $X$ for any $n\in \N$, where a vertex $x$ is  {\it $n$-level} if $|x|=n$. Then we have that
$$V=\bigcup_{n\in \N} V_n$$
is the set containing all  the vertices of the tree $X$.
For any vertex $x\in V$, denote by $I_x$ the set 
$$\{\xi\in \bx: \text{the geodesic $[0, \xi)$ passes through $x$}\}.$$ 
We denote by $\dyadic$ the set $\{I_x: x\in V\}$ and $\dyadic_n$ the set $\{I_x: x\in V_n\}$ for any $n\in \N$. Then $\dyadic_0=\{\partial X\}$ and we have
$$\dyadic =\bigcup_{n\in \N} \dyadic_n.$$
Then the set $\dyadic$ is a dyadic decomposition of $\bx$. Moreover, for any $n\in \N$ and $I\in \dyadic_{n}$, there is a unique element $\widehat I$ in $\dyadic_{n-1}$ such that $I$ is a subset of it. It is easy to see that if $I=I_x$ for some $x\in V_{n}$, then $\widehat I=I_y$ with $y$ the unique parent of $x$ in the tree $X$. Hence the structure of the tree $X$ gives a corresponding structure of the dyadic decomposition of $\partial X$ which we defined above.

Since we  want to characterize the trace spaces of the Newtonian spaces with respect to our measure $\mul$, we introduce the following Besov-type spaces  $\lbesov$.
\begin{defn}\label{besov-3}\rm
For $0\leq\theta<1$ and $p\geq 1$,  the Besov-type space $\lbesov$ consists of all functions $f\in L^p(\bx)$ for which the dyadic $\dlbesove$-energy of $f$ defined as
$$\|f\|^p_{\dlbesov}:=\sum_{n=1}^{\infty} e^{\epsilon n\theta p}n^\lambda \sum_{I\in \dyadic_n} \nu(I)\left|f_{I}-f_{ \widehat I}\right|^p$$
  is finite. The norm on $\lbesov$ is 
$$\|f\|_{\lbesov}:=\|f\|_{\lpb}+\|f\|_{\dlbesov}.$$
\end{defn}


%

Here and throughout this paper, the measure $\nu$ on the boundary $\bx$ is  the Ahlfors regular measure in Proposition \ref{Ahlfor-boundary} and $f_I$ is the  mean value $\dashint_I f\, d\nu =\frac{1}{\nu(I)} \int_I f\, d\nu$.

The following proposition states that the Besov space $\mathcal B^{\theta, \lambda}_p(\bx)$ coincides with the Besov space  $\Besov$ whenever $0<\theta<1$ and $\lambda=0$. The proof of this proposition follows by using \cite[Lemma 5.4]{BBGS} and a  modification of the proof of \cite[Proposition A.1]{KTW}. We omit the details.

\begin{prop}\label{norm-equiv}
Let $0<\theta<1$ and $p\geq 1$. For any $f\in L^1_{\rm loc}(\bx)$, we have
$$\|f\|_{\Besov}\approx \|f\|_{\mathcal B^{\theta, 0}_p(\bx)}.$$
\end{prop}

 For $\lambda>0$, we next define  special Besov-type spaces with $\theta=0$ and $p=1$. Before the definition,
  we first fix a sequence $\{\alpha(n): n\in \mathbb N\}$ such that there exist constants $c_1\geq c_0>1$ satisfying
\begin{equation}\label{sequence}
c_0\leq \frac{\alpha(n+1)}{\alpha(n)}\leq c_1, \ \ \ \forall \ n\in \mathbb N.
\end{equation}
A simple example of such a sequence is obtained by letting $\alpha(n)=2^{n}$.
\begin{defn}\label{besov-4}\rm
 For $\lambda>0$, the Besov-type space $\lbesovo$ consists of all functions $f\in L^1(\bx)$ for which the $\dlbesovoe$-dyadic energy of $f$ defined as
 $$\|f\|_{\dlbesovo}=\sum_{n=1}^{\infty} \alpha(n)^\lambda\sum_{I\in\dyadic_{\alpha(n)}}\nu(I)|f_I-f_{\widetilde I}|$$
  is finite. Here for any $I=I_x\in \dyadic_{\alpha(n)}$ with $x\in V_{\alpha(n)}$ and $n\geq 1$, we denote $\widetilde I=I_y$ where $y\in V_{\alpha(n-1)}$ is the ancestor of $x$ in $X$.
  The norm on $\lbesovo$ is 
$$\|f\|_{\lbesovo}:=\|f\|_{L^1(\bx)}+\|f\|_{\dlbesovo}.$$
\end{defn}

\begin{rem}\rm
Actually, the choice of the sequence $\{\alpha(n)\}_{n\in \mathbb N}$ will not affect the definition of $\lbesovo$: by Theorem \ref{th5} we  obtain that  any 
two choices of the sequences $\{\alpha(n)\}_{n\in \mathbb N}$  lead to comparable 
norms, for more details see Corollary \ref{comparable}.

\end{rem}

\section{Proofs}\label{proofs}

\subsection{Proof of Theorem \ref{th2} }
\begin{proof}
{\bf Trace Part:} Let $f\in \Newton$.  We first define the trace operator as
\begin{equation}\label{trace-operator}
\Tr f(\xi):=\tilde f(\xi)=\lim_{[0, \xi)\ni x\rarrow \xi} f(x), \ \ \xi\in \bx,
\end{equation}
where the limit is taken along the geodesic ray $[0, \xi)$. Then our task is to show that the above limit exists for $\nu$-a.e. $\xi\in \bx$ and that the trace $\Tr f$ satisfies the norm estimates.

Let $\xi\in \bx$ be arbitrary and let $x_j=x_j(\xi)$ be the ancestor of $\xi$ with $|x_j|=j$. 
To show that the limit in \eqref{trace-operator} exists for $\nu$-a.e. $\xi\in \bx$, it suffices to show that the function 
\begin{equation}\label{trace-operator1}
{\tilde f}^*(\xi)=|f(0)|+\sum_{j=0}^{+\infty} |f(x_{j+1})-f(x_{j})|
\end{equation}
is in $L^p(\bx)$, since if $\tilde f^*\in L^p(\bx)$, we have $|\tilde f^*|<\infty$ for $\nu$-a.e. $\xi\in \bx$.

Set $r_j=2e^{-j\epsilon}/\epsilon$. Then on the edge $[x_j, x_{j+1}]$ we have the relations
\begin{equation}\label{relation1}
ds\approx e^{(\beta-\epsilon)j} j^{-\lambda}\, d\mul\approx r_j^{1-\beta/\epsilon} j^{-\lambda}\, d\mu \ \ \ {\rm and}\ \ \mul([x_{j}, x_{j+1}])\approx r_j^{\beta/\epsilon} j^{\lambda},
\end{equation}
where the comparison constants depend on $\epsilon, \beta$.
Then we obtain the estimate
\begin{align}
{\tilde f}^*(\xi) &=|f(0)|+\sum_{j=0}^{+\infty}  |f(x_{j+1})-f(x_{j})|\leq |f(0)|+\sum_{j=0}^{+\infty} \int_{[x_{j}, x_{j+1}]} g_f \, ds\notag \\
&\lesssim |f(0)|+\sum_{j=0}^{+\infty}{r_j^{1-\beta/\epsilon}}j^{-\lambda}\int_{[x_{j}, x_{j+1}]} g_f \, d\mul\approx |f(0)|+\sum_{j=0}^{+\infty}{r_j}\dashint_{[x_{j}, x_{j+1}]} g_f \, d\mul, \label{estimate-f}
\end{align}
where $g_f$ is an upper gradient of $f$.

Since $\theta=1-(\beta-\log K)/(p\epsilon)>0$, we may choose $0<\kappa<\theta$. Then for $p>1$, by the H\"older inequality  and \eqref{relation1}, we have that
\begin{align*}
|{\tilde f}^*(\xi)|^p &\lesssim |f(0)|^p+\sum_{j=0}^{+\infty} r_j^{p(1-\kappa)}\dashint_{[x_{j}, x_{j+1}]} {g_f}^p \, d\mul\\
&\approx |f(0)|^p+\sum_{j=0}^{+\infty} r_j^{p(1-\kappa)-\beta/\epsilon} j^{-\lambda}\int_{[x_{j}, x_{j+1}]} {g_f}^p \, d\mul.
\end{align*}
For $p=1$, the above estimates are also true without using the H\"older inequality. It follows that for $p\geq 1$,
$$|{\tilde f}^*(\xi)|^p \lesssim |f(0)|^p+\sum_{j=0}^{+\infty} r_j^{p(1-\kappa)-\beta/\epsilon} j^{-\lambda}\int_{[x_{j}, x_{j+1}]} {g_f}^p \, d\mul.$$
Integrating over all $\xi\in \bx$, since $\nu(\bx)\approx 1$, we obtain by means of Fubini's theorem that
\begin{align*}
\int_\bx |{\tilde f}^*(\xi)|^p\, d\nu&\lesssim |f(0)|^p+\int_\bx \sum_{j=0}^{+\infty} r_j^{p(1-\kappa)-\beta/\epsilon} j^{-\lambda}\int_{[x_{j}(\xi), x_{j+1}(\xi)]} {g_f}^p \, d\mul\, d\nu(\xi)\\
&=|f(0)|^p +\int_X g_f(x)^p\int_\bx \sum_{j=0}^{+\infty} r_j^{p(1-\kappa)-\beta/\epsilon} j^{-\lambda}\chi_{[x_j(\xi), x_{j+1}(\xi)]}(x)\, d\nu(\xi)\, d\mul(x).
\end{align*}
Notice that $\chi_{[x_j(\xi), x_{j+1}(\xi)]}(x)$ is nonzero only if $j\leq |x|\leq j+1$ and $x<\xi$. Thus the last estimate can be rewritten as 
$$\int_\bx |{\tilde f}^*(\xi)|^p\, d\nu\lesssim |f(0)|^p + \int_X g_f(x)^p  r_{j(x)}^{p(1-\kappa)-\beta/\epsilon} j(x)^{-\lambda}\nu(E(x))\, d\mul(x),$$
where $E(x)=\{\xi\in \bx: x<\xi\}$ and $j(x)$ is the largest integer such that $j(x)\leq |x|$. 

It follows from \cite[Lemma 5.1]{BBGS} that $E(x)=B(\xi, r)$ for any $\xi\in E(x)$ and  $r\approx e^{-\epsilon j(x)}$. Hence we obtain from Proposition \ref{Ahlfor-boundary} that $\nu(E(x))\approx r_{j(x)}^Q$. Since $p(1-\kappa)>\beta/\epsilon-\log K/\epsilon=\beta/\epsilon-Q$, then for any $j(x)\in \mathbb N$, we have that
$$r_{j(x)}^{p(1-\kappa)-\beta/\epsilon+Q} j(x)^{-\lambda}\lesssim 1,$$
which induces the estimate
$$\int_\bx |{\tilde f}^*(\xi)|^p\, d\nu\lesssim |f(0)|^p+ \int_X g_f(x)^p\, d\mul(x).$$
Hence we obtain that ${\tilde f}^*$ is in $L^p(\bx)$, which gives the existence of the limit in \eqref{trace-operator} for $\nu$-a.e. $\xi\in\bx$. In particular, since $|\tilde f|\leq {\tilde f}^*$, we have the estimate
$$\int_\bx |\tilde f|^p\, d\nu\lesssim \int_X |f|^p\, d\mul + \int_X {g_f}^p\, d\mul,$$
and hence the norm estimate
\begin{equation}\label{L_p}
\|\tilde f\|_{L^p(\bx)}\lesssim \left(\int_X |f|^p\, d\mul + \int_X {g_f}^p\, d\mul\right)^{1/p}=\|f\|_{\Newton}.
\end{equation}

To estimate the dyadic energy $\|\tilde f\|^p_{\dlbesov}$, for any $I\in \dyadic_n$, $\xi\in I$ and $\zeta\in\widehat I$, 
we have that
$$|\tilde f(\xi)-\tilde f(\zeta)|\leq \sum_{j=n-1}^{+\infty}|f(x_j)-f(x_{j+1})|+\sum_{j=n-1}^{+\infty}|f(y_j)-f(y_{j+1})|,$$
where $x_j=x_j(\xi)$ and $y_j=y_j(\zeta)$ are the ancestors of $\xi$ and $\zeta$ with $|x_j|=|y_j|=j$, respectively. In the above inequality, we used the fact that $x_{n-1}(\xi)=y_{n-1}(\eta)$.
By using \eqref{relation1} and an  argument similar to \eqref{estimate-f}, we obtain that
$$|\tilde f(\xi)-\tilde f(\zeta)|\lesssim \sum_{j=n-1}^{+\infty}r_j \dashint_{[x_{j}(\xi), x_{j+1}(\xi)]} g_f \, d\mul+\sum_{j=n-1}^{+\infty}r_j \dashint_{[y_{j}(\zeta), y_{j+1}(\zeta)]} g_f \, d\mul.$$
Choose $0<\kappa<\theta$ and insert $r_j^{\kappa} r_j^{-\kappa}$ into the above sum. If $p>1$, then the H\"older inequality   and \eqref{relation1} imply that
\begin{align*}
|\tilde f(\xi)-\tilde f(\zeta)|^p&\lesssim r_{n-1}^{\kappa p} \sum_{j=n-1}^{+\infty}r_j^{p(1-\kappa)} \dashint_{[x_{j}(\xi), x_{j+1}(\xi)]} {g_f}^p \, d\mul+ r_{n-1}^{\kappa p}\sum_{j=n-1}^{+\infty}r_j^{p(1-\kappa)}  \dashint_{[y_{j}(\zeta), y_{j+1}(\zeta)]} {g_f}^p \, d\mul\\
&\approx r_{n-1}^{\kappa p} \sum_{j=n-1}^{+\infty}r_j^{p(1-\kappa)-\beta/\epsilon}j^{-\lambda} \left(\int_{[x_{j}(\xi), x_{j+1}(\xi)]} {g_f}^p \, d\mul+\int_{[y_{j}(\zeta), y_{j+1}(\zeta)]} {g_f}^p \, d\mul\right).
\end{align*}
For $p=1$ the estimates above is also true without  using the H\"older inequality.  
It follows from Fubini's theorem and from $\nu(I)\approx \nu(\widehat I)$ that
\begin{align*}
\sum_{I\in\dyadic_n}\nu(I)|\tilde f_I&-\tilde f_{\widehat I}|^p\leq \sum_{I\in \dyadic_n} \nu(I)\dashint_I\dashint_{\widehat I} |\tilde f(\xi)-\tilde f(\zeta)|^p\, d\nu(\xi)\, d\nu(\zeta)\\
&\lesssim \int_\bx r_{n-1}^{\kappa p} \sum_{j=n-1}^{+\infty}r_j^{p(1-\kappa)-\beta/\epsilon}j^{-\lambda} \int_{[x_{j}(\xi), x_{j+1}(\xi)]} {g_f}^p \, d\mul\, d\nu(\xi)\\
&=r_{n-1}^{\kappa p} \int_{X\cap\{|x|\geq n-1\}} {g_f}^p  \int_\bx \sum_{j=n-1}^{+\infty}r_j^{p(1-\kappa)-\beta/\epsilon}j^{-\lambda}\chi_{[x_j(\xi), x_{j+1}(\xi)]}(x)\, d\nu(\xi)\, d\mul(x).
\end{align*}
Using the notation $E(x)$ and $j(x)$ defined before, the above estimate can be rewritten as
\begin{align*}
\sum_{I\in\dyadic_n}\nu(I)|\tilde f_I-\tilde f_{\widehat I}| &\lesssim  r_{n-1}^{\kappa p} \int_{X\cap\{|x|\geq n-1\}} {g_f}^p  \, r_{j(x)}^{p(1-\kappa)-\beta/\epsilon}j(x)^{-\lambda}\nu(E(x))\, d\mul \\
&\lesssim  r_{n-1}^{\kappa p} \int_{X\cap\{|x|\geq n-1\}} {g_f}^p\, r_{j(x)}^{p(1-\kappa)-\beta/\epsilon+Q}j(x)^{-\lambda}\, d\mul.
\end{align*}
Since $e^{-\epsilon n}\approx r_{n-1}$ and $p-\beta/\epsilon+Q=\theta p$, we obtain the estimate
\begin{align*}
\|\tilde f\|^p_{\dlbesov}&\lesssim \sum_{n=1}^{+\infty} r_{n-1}^{\kappa p-\theta p} n^\lambda \int_{X\cap\{|x|\geq n-1\}} {g_f}^p\, r_{j(x)}^{p(1-\kappa)-\beta/\epsilon+Q}j(x)^{-\lambda}\, d\mul\\
&=\sum_{n=0}^{+\infty} r_{n}^{\kappa p-\theta p}(n+1)^\lambda \sum_{j=n}^{+\infty}\int_{X\cap\{j+1> |x|\geq j\}} {g_f}^p\, r_{j}^{(\theta -\kappa) p}j^{-\lambda}\, d\mul\\
&=\sum_{j=0}^{+\infty} \int_{X\cap\{j+1> |x|\geq j\}} {g_f}^p\, r_{j}^{(\theta -\kappa) p}j^{-\lambda}\, d\mul \left(\sum_{n=0}^{j} r_n^{\kappa p-\theta p}(n+1)^\lambda\right)\\
&\lesssim \sum_{j=0}^{+\infty} \int_{X\cap\{j+1> |x|\geq j\}} {g_f}^p\, d\mul=\int_X {g_f}^p\, d\mul.
\end{align*}
Here the last inequality employed  the estimate
$$\sum_{n=0}^{j} r_n^{\kappa p-\theta p}(n+1)^\lambda\lesssim r_j^{\kappa p-\theta p} (j+1)^\lambda\approx r_j^{(\kappa-\theta)p}j^\lambda,$$
which comes from the facts $r_n=2e^{-\epsilon n}/\epsilon$ and $\kappa p-\theta p<0$.
Thus, we obtain the estimate 
$$\|\tilde f\|_{\dlbesov}\lesssim \|g_f\|_{L^p(X, \mul)}\leq \|f\|_\Newton,$$
which together with \eqref{L_p} finishes the proof of Trace Part.

{\bf Extension Part:} Let $u\in \lbesov$.  For $x\in X$ with $|x|=n\in \mathbb N$, let 
\begin{equation}\label{extension-operator1}
\tilde u(x)=\dashint_{I_x} u\, d\nu,
 \end{equation} 
 where $I_x\in \dyadic_n$ is the set of all the points $\xi\in \bx$ such that the geodesic $[0, \xi)$ passes through $x$,  that is,  $I_x$ consists of all the points in $\bx$ that have $x$ as an ancestor. By \eqref{trace-operator} and \eqref{extension-operator1} we notice that $\Tr \tilde u(\xi)=u(\xi)$ whenever $\xi\in \bx$ is a Lebesgue point of $u$. 
 
 If $y$ is a child of $x$, then $|y|=n+1$ and $I_x$ is the parent of $I_y$. We extend $\tilde u$ to the edge $[x, y]$ as follows: For each $t\in [x, y]$, set
\begin{equation}\label{extension-operator2}
g_{\tilde u}(t)=\frac{\tilde u(y)-\tilde u(x)}{d_X(x, y)}=\frac{\epsilon (u_{I_y}-u_{I_x})}{(1-e^{-\epsilon})e^{-\epsilon n}}=\frac{\epsilon (u_{I_y}-u_{\widehat I_y})}{(1-e^{-\epsilon})e^{-\epsilon n}}
\end{equation}
 and
\begin{equation}\label{extension-operator3}
\tilde u(t)= \tilde u(x)+g_{\tilde u}(t)d_X(x, t).
\end{equation}
Then we define the extension of $u$ to be $\tilde u$. 

Since $g_{\tilde u}$ is a constant and $\tilde u$ is linear with respect to the metric $d_X$ on the edge $[x, y]$, it follows that $|g_{\tilde u}|$ is an upper gradient of $\tilde u$ on the edge $[x, y]$. We have that 
\begin{align}
\int_{[x, y]} |g_{\tilde u}|^p\, d\mul&\approx \int_{n}^{n+1} |u_{I_y}-u_{\widehat I_y}|^p e^{-\beta\tau+\epsilon n p}(\tau+C)^\lambda\, d\tau\notag\\
&\approx e^{(-\beta+\epsilon p)(n+1)} (n+1)^\lambda |u_{I_y}-u_{\widehat I_y}|^p.\label{lem2}
\end{align}
Now sum up the above integrals over all the edges on $X$ to obtain that
$$\int_X |g_{\tilde u}|^p\, d\mul \approx \sum_{n=1}^{+\infty}\sum_{I\in \dyadic_n} e^{(-\beta+\epsilon p) n}n^\lambda |u_I-u_{\widehat I}|^p.$$
For $I\in \dyadic_n$, the estimate
$$e^{\epsilon n\theta p}\nu(I)\approx e^{\epsilon n(p-(\beta-\log K)/\epsilon) -\epsilon n Q}\approx e^{n(\epsilon p-\beta)}$$
implies that
\begin{equation}\label{extension-norm}
\int_X |g_{\tilde u}|^p\, d\mul \approx \sum_{n=1}^{+\infty} e^{\epsilon n\theta p}n^\lambda \sum_{I\in \dyadic_n} \nu(I) |u_I-u_{\widehat I}|^p=\|u\|^p_{\dlbesov}.
\end{equation}
 
 To estimate the $L^p$-norm of $\tilde u$, we first observe that 
 \begin{equation}\label{L_p-relation}
 |\tilde u(t)|\leq |\tilde u(x)|+|g_{\tilde u}|d_X(x, y)=|\tilde u(x)|+|\tilde u(y)-\tilde u(x)|\lesssim |u_{I_x}|+|u_{I_y}|
 \end{equation}
for any $t\in [x, y]$. Then we obtain the estimate
 \begin{equation}\label{lem1}
 \int_{[x, y]} |\tilde u(t)|^p\, d\mul\lesssim \mul([x, y])\left(|u_{I_x}|^p+|u_{I_y}|^p\right)\lesssim e^{-\beta n+\epsilon n Q} n^\lambda\int_{I_x} |u|^p\, d\nu.
 \end{equation}
 Here the last inequality used the facts $\nu(I_x)\approx \nu(I_y)\approx e^{\epsilon n Q}$ and $\mul([x, y])\approx e^{-\beta n}n^\lambda$. Now sum up the above integrals over all the edges on $X$ to obtain that 
 $$\int_{X} |\tilde u(t)|^p\, d\mul \lesssim \sum_{n=0}^{+\infty}\sum_{I\in \dyadic_n}  e^{-\beta n+\epsilon n Q}n^\lambda \int_{I} |u|^p\, d\nu= \sum_{n=0}^{+\infty}e^{-\beta n+\epsilon n Q} n^\lambda \int_{\bx}  |u|^p\, d\nu.$$

Since $\beta-\epsilon Q=\beta -\log K>0$, the sum of  $e^{-\beta n+\epsilon n Q} n^\lambda$ converges. Hence we obtain the $L^p$-estimate
\begin{equation}\label{extension-L_p}
\int_{X} |\tilde u|^p\, d\mul \lesssim  \int_{\bx}  |u|^p\, d\nu.
\end{equation}

Combing  \eqref{extension-norm} with \eqref{extension-L_p}, we obtain the norm estimate
$$\|\tilde u\|_{\Newton}\lesssim \|u\|_{\lbesov}.$$
\end{proof}

\subsection{Proof of Theorem \ref{th3}}
\begin{prop}\label{prop3.1}
Let $p=(\beta-\log K)/\epsilon$ and $\lambda>p-1$ if $p>1$ or $\lambda\geq 0$ if $p=1$. Then the trace operator $\Tr$ defined in \eqref{trace-operator} is a bounded linear operator from $\Newton$ to $L^p(\bx)$.
\end{prop}
\begin{proof}
Let $f\in \Newton$. We first  show that the limit in \eqref{trace-operator} exists for $\nu$-a.e. $\xi\in \bx$. It suffices to show that the function ${\tilde f}^*$ defined by \eqref{trace-operator1} is in $L^p(\bx)$. By estimates \eqref{relation1} and \eqref{estimate-f}, we obtain that
$${\tilde f}^*(\xi) \lesssim  |f(0)|+\sum_{j=0}^{+\infty}{r_j}\dashint_{[x_{j}, x_{j+1}]} g_f \, d\mul.$$
Insert $j^{-\lambda/p}\, j^{\lbd/p}$ into the above sum. If $p>1$,  the  H\"older inequality gives us that
\begin{align*}
|{\tilde f}^*(\xi) |^p&\lesssim |f(0)|^p+\left(\sum_{j=0}^{+\infty} j^{\frac{-\lambda}{p}\cdot\frac{p}{p-1}}\right)^{p-1}\left(\sum_{j=0}^{+\infty} r_j^{p}j^{\lambda}\dashint_{[x_j, x_{j+1}]}  {g_f}^p \, d\mul\right)\\
&\lesssim  |f(0)|^p+ \sum_{j=0}^{+\infty} r_j^{p-\beta/\epsilon} \int_{[x_j, x_{j+1}]}  {g_f}^p \, d\mul,
\end{align*}
since  $\mu([x_j, x_{j+1}])\approx r_j^{\beta/\epsilon} j^\lambda$ and for $\lambda>p-1$, the sum $j^{-\lambda/(p-1)}$ converges.
If $p=1$, then the  H\"older inequality is not needed and the estimate is simpler. It follows that
$$|{\tilde f}^*(\xi) |^p  \lesssim |f(0)|^p+ \sum_{j=0}^{+\infty} r_j^{p-\beta/\epsilon} \int_{[x_j, x_{j+1}]}  {g_f}^p \, d\mul$$
for any $\lambda>p-1$ if $p=1$ or for $\lambda\geq 0$ if $p=1$. Integrating over all $\xi\in \bx$ we obtain by means of Fubini's theorem that
\begin{align*}
\int_\bx |{\tilde f}^*(\xi)|^p\, d\nu&\lesssim |f(0)|^p+\int_\bx \sum_{j=0}^{+\infty} r_j^{p-\beta/\epsilon} \int_{[x_{j}(\xi), x_{j+1}(\xi)]} {g_f}^p \, d\mul\, d\nu(\xi)\\
&=|f(0)|^p+\int_X g_f(x)^p \int_\bx \sum_{j=0}^{+\infty} r_j^{p-\beta/\epsilon} \chi_{[x_j(\xi), x_{j+1}(\xi)]}(x)\, d\nu(\xi)\, d\mul(x)\\
&\lesssim |f(0)|^p+\int_X g_f(x)^p r_{j(x)}^{p-\beta/\epsilon} \nu(E(x))\, d\mul(x)\\
&\lesssim |f(0)|^p+\int_{X}g_f(x)^p r_{j(x)}^{p-\beta/\epsilon+Q}\, d\mul(x)=|f(0)|^p+\int_{X}g_f(x)^p \, d\mul(x).
\end{align*}
Here in the above estimates, the notations $E(x)$ and $j(x)$ are the same ones as those we used in the proof of Theorem \ref{th2}. 
It follows that ${\tilde f}^*$ is in $L^p(\bx)$ with the estimate
$$\int_\bx |\tilde f|^p\, d\nu\lesssim \int_X |f|^p\, d\mul + \int_X {g_f}^p\, d\mul.$$
Hence the limit in the definition of our trace operator exists, i.e., the trace operator is well-defined, and we also have the estimate 
\begin{equation*}
\|\tilde f\|_{L^p(\bx)}\lesssim \left(\int_X |f|^p\, d\mul + \int_X {g_f}^p\, d\mul\right)^{1/p}=\|f\|_{\Newton},
\end{equation*}
which finishes the proof.
\end{proof}

\begin{example}\label{ex1}\rm
Let $f$ be the continuous function on $X$ given by $f(x)=\log(|x|+1)$. Then the function $g_f(x)=e^{\epsilon |x|}/(|x|+1)$ is an upper gradient of $f$ on $X$ with respect to the metric $d_X$. For $p=(\beta-\log K)/\epsilon >1$ and $\lambda=p-1-\delta$ with $\delta>0$ arbitrary, we have the estimates
$$\int_X {g_f}^p\, d\mul\approx \sum_{n=0}^{+\infty} \frac{e^{p\epsilon n}}{(n+1)^p} K^n e^{-\beta n} n^\lambda \approx \sum_{n=0}^{+\infty} \frac{e^{(p\epsilon-\beta+\log K) n}}{(n+1)^{1+\delta}}=\sum_{n=1}^{+\infty} \frac{1}{n^{1+\delta}}<\infty$$
and
$$\int_X |f|^p\, d\mul\approx \sum_{n=0}^{+\infty} \log^p(n+1) K^ne^{-\beta n}n^\lambda\approx \sum_{n=0}^{+\infty} e^{(-\beta+\log K)n} n^\lambda \log^p(n+1)<\infty.$$
\end{example}
Hence we have $f\in \Newton$. 
On the other hand, $f(x)\rightarrow \infty$ as $x\rarrow \bx$.

\begin{lem}\label{L_p>n}
Let $u\in L^1(\bx)$ and $\tilde u$ be defined by  \eqref{extension-operator1}, \eqref{extension-operator2} and \eqref{extension-operator3}. Then
$$\int_{X\cap\{|x|\geq n\}} |\tilde u|^p\, d\mu\lesssim r_n^{(\beta-\log K)/\epsilon} \int_\bx |u|^p\, d\nu,$$
where $n\in \mathbb N$, $p\geq 1$ and $r_n=2^{-n\epsilon}/\epsilon$.
\end{lem}
\begin{proof}
By using the estimate \eqref{L_p-relation}, for $x, y\in X$ with $y$ a child of $x$ and $|x|=j$, we obtain that
$$\int_{[x, y]} |\tilde u(t)|^p\, d\mu\lesssim \mu([x, y]) (|u_{I_x}|^p+|u_{I_x}|^p)\lesssim e^{-\beta j+\epsilon j Q}\int_{I_x} |u|^p\, d\nu.$$
Summing up the integrals over all edges of $X\cap\{|x|\geq n\}$, we obtain that
\begin{align*}
\int_{X\cap\{|x|\geq n\}} |\tilde u|^p\, d\mu&\lesssim \sum_{j=n}^{+\infty} \sum_{I\in \dyadic_j}  e^{-\beta j+\epsilon j Q}  \int_{I} |u|^p\, d\nu= \sum_{j=n}^{+\infty}e^{-\beta j+\epsilon j Q}   \int_{\bx}  |u|^p\, d\nu \\
& \approx e^{-(\beta-\log K) n}\int_{\bx}  |u|^p\, d\nu \approx r_n^{(\beta-\log K)/\epsilon} \int_\bx |u|^p\, d\nu.
\end{align*}
\end{proof}

\begin{lem}\label{Lip>n}
Let $u$ be  Lipschitz continuous on $\bx$ and $\tilde u$ be defined by \eqref{extension-operator1}, \eqref{extension-operator2} and \eqref{extension-operator3}. Then
$$\int_{X\cap\{|x|\geq n\}}|g_{\tilde u}|^p\, d\mu\lesssim  r_n^{(\beta-\log K)/\epsilon}\LIP(u, \bx)^p,$$
where $r_n=2e^{-n\epsilon}/\epsilon$, $p\geq 1$ and 
$$\LIP(u, \bx)=\sup_{\xi, \zeta\in \bx: \xi\not=\zeta}\frac{|u(\xi)-u(\zeta)|}{d_X(\xi, \zeta)}.$$
\end{lem}
\begin{proof}
For $x, y\in X$ with $y$ a child of $x$ and $|x|=j$, since $g_{\tilde u}$ is a constant  on the edge $[x, y]$, we obtain the estimate 
$$\int_{[x, y]} |g_{\tilde u}|^p\, d\mu \approx \int_j^{j+1} \frac{ |u_{I_y}-u_{\widehat I_y}|^p}{e^{-\epsilon jp}} e^{-\beta\tau}\, d\tau\approx e^{-\beta j+\epsilon jp}|u_{I_y}-u_{\widehat I_y}|^p .$$

Summing up the above integrals over all edges of $X\cap\{|x|\geq n\}$, we obtain that
$$\int_{X\cap\{|x|\geq n\}} |g_{\tilde u}|^p\, d\mu \approx \sum_{j=n+1}^{+\infty}\sum_{I\in \dyadic_j} e^{(-\beta+\epsilon p ) j} |u_I-u_{\widehat I}|^p.$$
Since $u$ is Lipschitz on $\bx$, then for any $\xi, \zeta \in \bx$,
$$|f(\xi)-f(\zeta)|\leq \LIP(u, \bx)d_X(\xi, \zeta).$$
Hence, for any $I\in \dyadic_j$, we have that
\begin{align*}
|u_I-u_{\widehat I}|^p&\lesssim \dashint_I\dashint_{\widehat I} |f(\xi)-f(\zeta)|^p\,d\nu(\xi)\, d\nu(\zeta)\leq \dashint_I\dashint_{\widehat I}\LIP(u, \bx)^p d_X(\xi, \zeta)^p\,d\nu(\xi)\, d\nu(\zeta)\\
&\leq \LIP(u, \bx)^p \diam\big(\widehat I\big)^p\approx e^{-j\epsilon p} \LIP(u, \bx)^p.
\end{align*}
It follows that
\begin{align*}
\int_{X\cap\{|x|\geq n\}}|g_{\tilde u}|^p\, d\mu&\lesssim \sum_{j=n+1}^{+\infty} K^j e^{(-\beta+\epsilon p) j} e^{-j\epsilon p}\LIP(u, \bx)^p\\
& =\sum_{j=n+1}^{+\infty} e^{-(\beta-\log K) j}\LIP(u, \bx)^p \\
&\approx e^{-(\beta-\log K) n} \LIP(u, \bx)^p\approx  r_n^{(\beta-\log K)/\epsilon}\LIP(u, \bx)^p.
\end{align*}
\end{proof}

\begin{prop}\label{prop3.2}
Let $p=(\beta-\log K)/\epsilon\geq 1$. Then there exists a bounded non-linear extension operator $\Ex$ from $L^p(\bx)$ to $N^{1, p}(X)$ that acts as a right inverse of the trace operator $\Tr$ in \eqref{trace-operator}, i.e., $\Tr \circ\Ex=\Id$ on $L^p(\bx)$.
\end{prop}

The construction of the extension operator is given by gluing the $N^{1, p}$ extensions in Lemma \ref{Lip>n} of Lipschitz approximations of the boundary data with respect to a sequence of layers on the tree $X$. The main idea of the construction is inspired by \cite[Section 7]{Ma} and \cite [Section 4]{MSS} whose core ideas can be traced back to Gagliardo \cite{Ga} who discussed extending functions in $L^{1}(\real^n)$  to $W^{1, 1}(\real^{n+1}_+)$.

\begin{proof}[Proof of Proposition \ref{prop3.2}]
Let $f\in L^p(\bx)$. We approximate $f$ in $L^p(\bx)$ by a sequence of Lipschitz functions $\{f_k\}_{k=1}^{+\infty}$ such that $\|f_{k+1}-f_k\|_{L^p(\bx)} \leq 2^{2-k} \|f\|_{L^p(\bx)}$. Note that this requirement of rate of convergence of $f_k$ to $f$ ensures that $f_k\rarrow f$ pointwise $\nu$-a.e. in $\bx$. For technical reasons, we choose $f_1\equiv 0$. 

Then we choose a decreasing sequence of real numbers $\{\rho_k\}_{k=1}^{+\infty}$ such that

\quad\quad\noindent $\bullet$ $\rho_k\in\{e^{-\epsilon n}/\epsilon: n\in \mathbb N\}$;

\quad\quad\noindent $\bullet$  $0<\rho_{k+1}\leq \rho_k/2$;

\quad\quad\noindent $\bullet$  $\sum_{k} \rho_k \LIP(f_k, \bx)\leq C\|f\|_{L^p(\bx)}$.\\
These will now be used to define layers in $X$. Let 
$$\psi_k(x)=\max\left\{0, \min\left\{1, \frac{\rho_k-\dist(x, \bx)}{\rho_k-\rho_{k+1}}\right\}\right\}, \ \ x\in X.$$
We denote $-\log(\epsilon \rho_k)/\epsilon$ by $[\rho_k]$. This is a integer satisfying $e^{-\epsilon [\rho_k]}/\epsilon =\rho_k$. Then we obtain $0\leq\psi_k\leq 1$ and that
\begin{equation} \label{psi}
\psi_k(x)=\left\{\begin{array}{cl}
0, &\ \ |x|\leq [\rho_k];\\
1, & \ \ |x|\geq [\rho_{k+1}].
\end{array}\right.
\end{equation}
For any Lipschitz function $f_k$, we can define the extension $\tilde f_k$ of $f_k$ by using \eqref{extension-operator1}, \eqref{extension-operator2} and \eqref{extension-operator3}. Then we define the extension of $f$ as
\begin{equation}\label{extension-L_1}
\tilde f(x) :=\sum_{k=2}^{+\infty}(\psi_{k-1}(x)-\psi_k(x))\tilde f_k(x)=\sum_{k=1}^{+\infty}{\psi_k(x)}(\tilde f_{k+1}(x)-\tilde f_k(x)).
\end{equation}
It follows from \eqref{psi} that for any $x\in X$ with $|x|=[\rho_k]$, we have $\tilde f(x)=\tilde f_{k-1}(x)$. Since for the trace operator $\Tr$ defined in \eqref{trace-operator}, $\Tr \tilde f_k= f_k$ for $\nu$-a.e. in $\bx$, the pointwise convergence  $f_k\rarrow f$  $\nu$-a.e. in $\bx$ implies that $\Tr \tilde f=f$ for $\nu$-a.e. in $\bx$, since $\{[\rho_k]\}_{k=1}^{+\infty}$ is a subsequence of $\mathbb N$. Hence the extension operator defined by \eqref{extension-L_1} is a right inverse of the trace operator $\Tr$ in \eqref{trace-operator}.

It remains to show that $\tilde f \in N^{1, p}(X)$ with norm estimates. Lemma \ref{L_p>n} allows us to obtain the $L^p$-estimate for $\tilde f$. Since the extension operator that we apply for each $f_k$ is linear, we have that $\tilde f_{k+1}- \tilde f_k=\widetilde{ f_{k+1}-  f_k}$. Therefore, it follows from $(\beta-\log K)/\epsilon=p$ that 
\begin{align*}
\|\tilde f\|_{L^p(X)}&\leq \sum_{k=1}^{+\infty}\|{\psi_k}(\tilde f_{k+1}-\tilde f_k)\|_{L^p(X)}\leq \sum_{k=1}^{+\infty} \|\tilde f_{k+1}-\tilde f_k\|_{L^p\left(X\cap\{|x|\geq [\rho_k]\}\right)}\\
& \lesssim \sum_{k=1}^{+\infty}  r_{[\rho_k]} \|f_{k+1}- f_k\|_{L^p(\bx)}\approx \sum_{k=1}^{+\infty}  \rho_k \|f_{k+1}- f_k\|_{L^p(\bx)}\\
&\lesssim \sum_{k=1}^{+\infty} \|f_{k+1}- f_k\|_{L^p(\bx)}\lesssim \|f\|_{L^p(\bx)}.
\end{align*}

In order to obtain the $L^p$-estimate of an upper gradient of $\tilde f$, it suffices to consider the $L^p$-estimate of $\Lip \tilde f$, where for any function $u$, $\Lip u(x)$ is defined as
$$\Lip u(x)=\limsup_{y\rarrow x} \frac{|u(y)-u(x)|}{d_X(x, y)}.$$
We first apply the product rule for locally Lipschitz function, which yields that
\begin{align*}
\Lip \tilde f&=\sum_{k=1}^{+\infty} \left(|\widetilde{ f_{k+1}-  f_k}|\Lip \psi_k+\psi_k\Lip (\widetilde{ f_{k+1}-  f_k})\right)\\
&\leq\sum_{k=1}^{+\infty} \left( \frac{|\widetilde{ f_{k+1}-  f_k}|\chi_{\{|x|\geq [\rho_k]\}}}{\rho_k-\rho_{k+1}}+\chi_{\{|x|\geq [\rho_k]\}}\Lip (\widetilde{ f_{k+1}-  f_k}) \right).
\end{align*}
 Thus, 
\begin{align*}
\|\Lip \tilde f\|_{L^p(\bx)} \leq \sum_{k=1}^{+\infty} \left( \left\|\frac{|\widetilde{ f_{k+1}-  f_k}|}{\rho_k-\rho_{k+1}}\right\|_{L^p(X\cap\{|x|\geq [\rho_k]\})} +\|\Lip (\widetilde{ f_{k+1}-  f_k})\|_{L^p(X\cap\{|x|\geq [\rho_k]\})}\right).
\end{align*}
It follows from Lemma \ref{L_p>n} that
\begin{align*}
\sum_{k=1}^{+\infty} \left\|\frac{|\widetilde{ f_{k+1}-  f_k}|}{\rho_k-\rho_{k+1}}\right\|_{L^p(X\cap\{|x|\geq [\rho_k]\})}&\lesssim \sum_{k=1}^{+\infty} \frac{\rho_k}{\rho_k-\rho_{k+1}} \| f_{k+1}-  f_k\|_{L^p(\bx)}\\
&\approx  \sum_{k=1}^{+\infty}\| f_{k+1}-  f_k\|_{L^p(\bx)}\lesssim \|f\|_{L^p(\bx)}.
\end{align*}
Recall that $\tilde u$ is affine one any edge of $X$, with ``slope''  $g_{\tilde u}$, for the extension $\tilde u$ given via  \eqref{extension-operator1}, \eqref{extension-operator2} and \eqref{extension-operator3}, for any function $u$. Hence  $\Lip \tilde u=g_{\tilde u}$. Therefore, it follows from  Lemma \ref{Lip>n} that
\begin{align*}
\sum_{k=1}^{+\infty} \|\Lip (\widetilde{ f_{k+1}-  f_k})\|_{L^p(X\cap\{|x|\geq [\rho_k]\})}&\lesssim \sum_{k=1}^{+\infty}  \rho_k \LIP(f_{k+1}-f_k, \bx)\\
&\leq\sum_{k=1}^{+\infty} \rho_k\left(\LIP(f_{k+1}, \bx)+\LIP(f_k, \bx)\right)\\
&\lesssim \|f\|_{L^p(\bx)}.
\end{align*}
Here in the last inequality, we used the defining properties of $\{\rho_k\}_{k=1}^{+\infty}$. Thus, we have shown that
$$\|\Lip \tilde f\|_{L^p(\bx)} \lesssim \|f\|_{L^p(\bx)}.$$

Altogether, we obtain that
$$\|\tilde f\|_{N^{1, p}(X)}\leq\| \tilde f\|_{L^p(\bx)}+\|\Lip \tilde f\|_{L^p(\bx)}\lesssim \|f\|_{L^p(\bx)}.$$
\end{proof}

\begin{proof}[Proof of Theorem \ref{th3}]
The boundedness and linearity of the trace operator follows from Proposition \ref{prop3.1} and the sharpness of $\lambda>p-1$ follows from Example \ref{ex1}. The extension operator is given in Proposition \ref{prop3.2}.
\end{proof}

\begin{rem}\rm
For  $p=(\beta-\log K)/\epsilon>1$ and $\lambda>p-1$, Theorem \ref{th3} only tells us that there exists a bounded linear trace operator \eqref{trace-operator} from $N^{1, p}(X, \mu_\lambda)$ to $L^p(\bx)$. It is unknown whether this trace operator is surjective or not. All we know is that there exists a nonlinear bounded extension operator from $L^p(\bx)$ to $N^{1, p}(X)$ that acts as a right inverse of the trace operator \eqref{trace-operator}. Since $\lambda>p-1>0$ 
implies $N^{1, p}(X, \mu_\lambda)\subsetneq N^{1, p}(X)$, we have an open question: Which space does the  bounded linear trace operator \eqref{trace-operator} map $N^{1, p}(X, \mu_\lambda)$ surjectively onto?
\end{rem}

\subsection{Proof of Theorem \ref{th5}}
\begin{proof}[Proof of Theorem \ref{th5}]
{\bf Trace Part:} Let $f\in N^{1,1}(X, \mul)$ with $\lambda>0$ and let $g_f$ be an upper gradient of $f$.
By Proposition \ref{prop3.1}, we know that the trace operator $\Tr f=\tilde f$ defined in \eqref{trace-operator} is well-defined and that $\tilde f$ satisfies the norm estimate
$$\|\tilde f\|_{L^1(\bx)}\lesssim \|f\|_{N^{1,1}(X, \mul)}.$$
Then the remaining task is to establish the estimate on the dyadic energy $\|\tilde f\|_{\dlbesovo}$. For any $I\in \dyadic_{\alpha(n)}$, $\xi\in I$ and $\zeta\in \widetilde I\in \dyadic_{\alpha(n-1)}$, we obtain that
\begin{align*}
|\tilde f(\xi)-\tilde f(\zeta)|&\leq \sum_{j=\alpha(n-1)}^{+\infty}|f(x_j)-f(x_{j+1})|+\sum_{j=\alpha(n-1)}^{+\infty}|f(y_j)-f(y_{j+1})|\\
&\lesssim  \sum_{j=\alpha(n-1)}^{+\infty}r_j\dashint_{[x_j, x_{j+1}]} g_f\, d\mul+ \sum_{j=\alpha(n-1)}^{+\infty}r_j\dashint_{[y_j, y_{j+1}]} g_f\, d\mul,
\end{align*}
where $x_j=x_j(\xi)$ and $y_j=y_j(\zeta)$ are the ancestors of $\xi$ and $\zeta$ with $|x_j|=|y_j|=j$, respectively. For any $I\in \dyadic_{\alpha(n)}$ and any function $h\in L^1(\bx)$, we have 
$$\frac{\nu(I)}{\nu(\widetilde I)}\approx \left(\frac{r_{\alpha(n)}}{r_{\alpha(n-1)}}\right)^Q\approx e^{(\alpha(n-1)-\alpha(n))\log K}\approx K^{\alpha(n-1)-\alpha(n)}$$
and 
\begin{equation}\label{n-n+1}
\sum_{I\in \dyadic_{\alpha(n)}} \int_{\widetilde I} h(\zeta)\, d\nu(\zeta)=K^{\alpha(n)-\alpha(n-1)} \int_{\bx} h(\zeta)\, d\nu(\zeta).
\end{equation}
Hence it follows from  the fact that $\mul([x_j, x_{j+1}])\approx r_j^{\beta/\epsilon}j^\lambda$ and Fubini's theorem that
\begin{align*}
\sum_{I\in\dyadic_{\alpha(n)}}\nu(I)|\tilde f_I-&\tilde f_{\widetilde I}| \leq  \sum_{I\in \dyadic_{\alpha(n)}} \nu(I) \dashint_I\dashint_{\widetilde I} |\tilde f(\xi)-\tilde f(\zeta)|\, d\nu(\xi)\, d\nu(\zeta)\\
& \lesssim \sum_{I\in \dyadic_{\alpha(n)}} \int_I  \sum_{j=\alpha(n-1)}^{+\infty}r_j\dashint_{[x_j(\xi), x_{j+1}(\xi)]} g_f\, d\mul\, d\nu(\xi)\\ 
&\ \ \ \ \ \ \ \ \ +\sum_{I\in \dyadic_{\alpha(n)}} K^{\alpha(n-1)-\alpha(n)} \int_{\widetilde I} \sum_{j=\alpha(n-1)}^{+\infty}r_j\dashint_{[y_j(\zeta), y_{j+1}(\zeta)]} g_f\, d\mul\, d\nu(\zeta)\\
& \approx \int_{\bx} \sum_{j=\alpha(n-1)}^{+\infty}r_j\dashint_{[x_j(\xi), x_{j+1}(\xi)]} g_f\, d\mul\, d\nu(\xi)\\
&\approx\int_{X\cap\{|x|\geq \alpha(n-1)\}} g_f \int_\bx \sum_{j=\alpha(n-1)}^{+\infty} r_j^{1-\beta/\epsilon} j^{-\lambda}\chi_{[x_j(\xi), x_{j+1}(\xi)]}(x) \, d\nu(\xi)\, d\mul(x).
 \end{align*}
Using the notation $E(x)$ and $j(x)$ defined in the proof of Theorem \ref{th2}, the above estimate can be rewritten as
\begin{align*}
\sum_{I\in\dyadic_{\alpha(n)}}\nu(I)|\tilde f_I-\tilde f_{\widetilde I}| &\lesssim \int_{X\cap\{|x|\geq \alpha(n-1)\}} g_f r_{j(x)}^{1-\beta/\epsilon}j(x)^{-\lambda} \nu(E(x))\, d\mul\\
& \lesssim  \int_{X\cap\{|x|\geq \alpha(n-1)\}} g_f r_{j(x)}^{1-\beta/\epsilon+Q}j(x)^{-\lambda}\, d\mul\\
&=\int_{X\cap\{|x|\geq \alpha(n-1)\}} g_f j(x)^{-\lambda}\, d\mul.
\end{align*}
It follows that
\begin{align*}
\sum_{n=1}^{\infty} \alpha(n)^\lambda\sum_{I\in\dyadic_{\alpha(n)}}\nu(I)|f_I-f_{\widetilde I}| &\lesssim \sum_{n=1}^{\infty}\alpha(n)\sum_{j=\alpha(n-1)}^{+\infty}\int_{X\cap\{j+1>|x|\geq j\}} g_f j^{-\lambda}\, d\mul\\
&=\sum_{n=0}^{\infty}\alpha(n+1)\sum_{j=\alpha(n)}^{+\infty}\int_{X\cap\{j+1>|x|\geq j\}} g_f j^{-\lambda}\, d\mul\\
&\leq \sum_{j=0}^{+\infty} \int_{X\cap\{j+1>|x|\geq j\}} g_f j^{-\lambda}\, d\mul\left(\sum_{n=0}^{\alpha^{-1}(j)}\alpha(n+1)^\lambda\right),
\end{align*}
 where $\alpha^{-1}(j)$ is the largest integer $m$ such that $\alpha(m)\leq j$. Since $\lambda>0$ and
$$1<c_0\leq \frac{\alpha(n+1)}{\alpha(n)}\leq c_1,$$
we obtain the estimate
$$\sum_{n=0}^{\alpha^{-1}(j)}\alpha(n+1)^\lambda\approx \sum_{n=0}^{\alpha^{-1}(j)}\alpha(n)^\lambda \leq \sum_{k=0}^{+\infty} j^\lambda c_0^{-\lambda k} \lesssim j^\lambda.$$
Hence we obtain the estimate 
\begin{align*}
\|\tilde f\|_{\dlbesovo}=\sum_{n=1}^{\infty} \alpha(n)^\lambda\sum_{I\in\dyadic_{\alpha(n)}}\nu(I)|f_I-f_{\widetilde I}| &\lesssim \sum_{j=0}^{+\infty} \int_{X\cap\{j+1>|x|\geq j\}} g_f\, d\mul\\
&=\int_X g_f\, d\mul=\|g_f\|_{L^1(X, \mul)}.
\end{align*}
Thus, we obtain the norm estimate 
$$\|f\|_{\lbesovo}=\|f\|_{L^1(\bx)}+\|f\|_{\dlbesovo}\lesssim \|f\|_{N^{1,1}(X, \mul)},$$
which finishes the proof of the Trace Part.

{\bf Extension Part:} Let $u\in \lbesovo$. Since $\alpha(0)$ is not necessarily zero, we let $\alpha(-1)=0$. 
For any $x\in X$ with $|x|=\alpha(n)$ and $-1\leq n\in \mathbb Z$, let 
$$\tilde u(x)=\dashint_{I_x} u\, d\nu,$$
where  $I_x\in\dyadic$ is the set of all the points $\xi\in\bx$ such that the geodesic $[0, \xi)$ passes through $x$, that is, $I_x$ consists of all the points in $\bx$ that have $x$ as an ancestor. 
 
If $y$ is a descendant of $x$ with $|y|=\alpha(n+1)$, then there exists $\tilde y\in X$ which is the parent of $y$. We extend $\tilde u$ to the edge $[x, y]$ as follows: For each $t\in [x, \tilde y]$, set $\tilde u(t)=\tilde u(x)$ and $g_{\tilde u}(t)=0$; for each $t\in [\tilde y, y]$, set
$$g_{\tilde u}(t)= \frac{\tilde u(y)-\tilde u(x)}{d_X(\tilde y, y)}=\frac{\epsilon (u_{I_y}-u_{I_x})}{(e^{\epsilon}-1)e^{-\epsilon \alpha(n+1)}}=\frac{\epsilon (u_{I_y}-u_{\widetilde I_y})}{(e^{\epsilon}-1)e^{-\epsilon \alpha(n+1)}}$$
and
$$\tilde u(t)=\tilde u(x)+ g_{\tilde u}(t) d_X(\tilde y, t).$$
Then we define $\tilde u$ to be the extension of $u$. Notice that $\Tr \tilde u(\xi)=u(\xi)$ whenever $\xi$ is a Lebesgue point of $u$. 

Now on the geodesic $[x, \tilde y]$, $g_{\tilde u}$ is zero and $\tilde u$ is a constant; on the edge $[\tilde y, y]$, $g_{\tilde u}$ is a constant and $\tilde u$ is linear with respect to the metric on the edge $[\tilde y, x]$. It follows that  $|g_{\tilde u}|$ is an upper gradient of $\tilde u$ on the geodesic $[x, y]$.  Then for $x\in X$ with $|x|=\alpha(n)$, $n\geq 0$, we obtain the estimate
\begin{align}
\int_{[x, y]} |g_{\tilde u}|\, d\mul =\int_{[\tilde y, y]}  |g_{\tilde u}|\, d\mul &\approx \int_{\alpha(n+1)-1}^{\alpha(n+1)} \frac{|u_{I_y}-u_{\widetilde I_y}|}{e^{-\epsilon \alpha(n+1)}} e^{-\beta\tau} (t+C)^\lambda\, d\tau\notag\\
&\approx e^{(\epsilon-\beta)\alpha(n+1)}\alpha(n+1)^\lambda |u_{I_y}-u_{\widetilde I_y}|.\label{add-1}
\end{align}
For $x=0$ and $|y|=\alpha(0)$, since $\nu(I_0)\approx \nu(I_y)\approx 1$, we have the estimate
\begin{equation}\label{add-2}
\int_{[0, y]}  |g_{\tilde u}|\, d\mul =\int_{[\tilde y, y]}  |g_{\tilde u}|\, d\mul  \approx |u_{I_0}-u_{I_y}|\leq |u_{I_0}|+|u_{I_y}|\lesssim \int_\bx |u|\, d\nu.
\end{equation}
Now sum up the estimates \eqref{add-1} and \eqref{add-2} over all edges of $X$ to obtain that
\begin{align*}
\int_X  |g_{\tilde u}|\, d\mul&=\int_{X\cap\{|x|\leq \alpha(0)\}} |g_{\tilde u}|\, d\mul+\int_{X\cap\{|x|\geq \alpha(0)\}} |g_{\tilde u}|\, d\mul\\
&\lesssim \sum_{y\in V_{\alpha(0)}} \int_{[0, y]}  |g_{\tilde u}|\, d\mul +\sum_{n=1}^{+\infty}\sum_{y\in V_{\alpha(n)} } \int_{[x, y]} |g_{\tilde u}|\, d\mul\\
&\lesssim K^{\alpha(0)} \int_\bx |u|\, d\nu+\sum_{n=1}^{+\infty}\sum_{I\in \dyadic_{\alpha(n)}}e^{(\epsilon-\beta)\alpha(n)}\alpha(n)^\lambda |u_{I}-u_{\widetilde I}|.
\end{align*}
Since for any $I\in \dyadic_{\alpha(n)}$, we have that
$$\nu(I) \approx r_{\alpha(n)}^Q\approx e^{-\epsilon \alpha(n) \log K/\epsilon}=e^{-\alpha(n) \log K}=e^{(\epsilon-\beta)\alpha(n)}.$$
Hence we obtain the estimate
 \begin{align}
 \int_X  |g_{\tilde u}|\, d\mul&\lesssim  \int_\bx |u|\, d\nu+ \sum_{n=1}^{\infty} \alpha(n)^\lambda\sum_{I\in\dyadic_{\alpha(n)}}\nu(I)|f_I-f_{\widetilde I}|\notag\\
& =\|u\|_{L^1(\bx)}+\|u\|_{\dlbesovo}=\|u\|_{\lbesovo}.\label{extension-p=1}
\end{align}

Towards the $L^1$-estimate for $\tilde u$, by the construction, we know that $|\tilde u(t)|=|\tilde u(x)|$ on the geodesic $[x, \tilde y]$ and that $|\tilde u(t)|\lesssim |\tilde u(x)|+|\tilde u(y)|$ on the edge $[\tilde y, y]$. Then for $n\geq -1$,  we have the estimate 
\begin{align*}
&\ \ \ \ \ \int_{X\cap\{\alpha(n)\leq |x|\leq \alpha(n+1)\}} |\tilde u|\, d\mul\\
&=\int_{X\cap\{\alpha(n)\leq |x|\leq \alpha(n+1)-1\}} |\tilde u|\, d\mul+\int_{X\cap\{\alpha(n+1)-1\leq |x|\leq \alpha(n+1)\}} |\tilde u|\, d\mul\\
&\leq \sum_{x\in V_{\alpha(n)}} |u(x)| \mul(F(x, d_X(x, \bx)))+ \sum_{y\in V_{\alpha(n+1)}} (|\tilde u(x)|+|\tilde u(y)|) \mul([\tilde y, y])=:H^n_1+H^n_2.
\end{align*}
By Lemma \ref{lF}, we obtain the estimate 
$$H^n_1 \lesssim \sum_{x\in V_{\alpha(n)}} e^{(-\beta +\log K)\alpha(n)}\alpha(n)^\lambda \int_{I_x} |u|\, d\nu= e^{(-\beta +\log K)\alpha(n)}\alpha(n)^\lambda\int_\bx |u|\, d\nu.$$
For $H^n_2$, by \eqref{n-n+1} and relation \eqref{relation1}, we have that
\begin{align*}
H^n_2&\lesssim \sum_{y\in V_{\alpha(n+1)}} e^{(-\beta+\log K)\alpha(n+1)} \alpha(n+1)^\lambda\left( \int_{I_y} |u|\, d\nu+ K^{\alpha(n)-\alpha(n+1)}\int_{\widetilde I_y} |u|\, d\nu\right)\\
&\lesssim e^{(-\beta+\log K)\alpha(n+1)} \alpha(n+1)^\lambda \int_\bx |u|\, d\nu.
\end{align*}
Sum up the above estimate with respect to $n$ to obtain via $\epsilon= \beta-\log K$ that
\begin{align}
\int_X |\tilde u|\, d\mul &=\sum_{n=-1}^{+\infty} \int_{X\cap\{\alpha(n)\leq |x|\leq \alpha(n+1)\}} |\tilde u|\, d\mul=\sum_{n=-1}^{+\infty} H^n_1+H^n_2\notag\\
&\lesssim\sum_{n=-1}^{+\infty}e^{(-\beta +\log K)\alpha(n)}\alpha(n)^\lambda\int_\bx |u|\, d\nu\notag\\
&=\sum_{n=-1}^{+\infty} e^{-\epsilon \alpha(n)}\alpha(n)^\lambda \int_\bx |u|\, d\nu\lesssim  \int_\bx |u|\, d\nu=\|u\|_{L^1(\bx)}.\label{extension-p-1}
\end{align}

By the  estimates \eqref{extension-p=1} and \eqref{extension-p-1}, we obtain the norm estimate
$$\|\tilde u\|_{N^{1, 1}(X, \mul)}\lesssim \|u\|_{\lbesovo}.$$

\end{proof}

\begin{cor}\label{comparable}
For given sequences $\{\alpha_1(n)\}_{n\in \N}$ and $\{\alpha_1(n)\}_{n\in \N}$ satisfying the relation \eqref{sequence} with respect to different pairs of $(c_0, c_1)$, the Banach spaces ${\mathcal B}^{0, \lambda}_{\alpha_1}(\bx)$ and ${\mathcal B}^{0, \lambda}_{\alpha_2}(\bx)$ coincide.
\end{cor}
\begin{proof}
For any function $u\in {\mathcal B}^{0, \lambda}_{\alpha_1}(\bx)$, by the Extension part in the proof of Theorem \ref{th5}, there is an extension $Eu=\tilde u$ such that
$$\|\tilde u\|_{N^{1,1}(X, \mul)}\lesssim \|u\|_{{\mathcal B}^{0, \lambda}_{\alpha_1}(\bx)}.$$
Since $u=T\circ Eu=T(\tilde u)$, it follows from the trace part in the proof of Theorem \ref{th5} that we have the estimate
$$\|u\|_{{\mathcal B}^{0, \lambda}_{\alpha_2}(\bx)}\lesssim \|\tilde u\|_{N^{1,1}(X, \mul)}.$$
Thus, we obtain
$$\|u\|_{{\mathcal B}^{0, \lambda}_{\alpha_2}(\bx)}\lesssim \|u\|_{{\mathcal B}^{0, \lambda}_{\alpha_1}(\bx)}.$$
The opposite inequality follows analogously and the claim follows.
\end{proof}

Next, we compare the function spaces $\lbesovo$ and ${\mathcal B}^{0, \lambda}_1(\bx)$.

\begin{prop}\label{compare}
Let $\lambda>0$. The space ${\mathcal B}^{0, \lambda}_1(\bx)$ is a subset of $\lbesovo$, i.e., for any $f\in L^1 (\bx)$, we have
$$\|f\|_{\dlbesovo}\lesssim \|f\|_{\dot{\mathcal B}^{0, \lambda}_1(\bx)}.$$
\end{prop}
\begin{proof}
Let $f\in L^1(\bx)$. For any $I\in \dyadic_{\alpha(n)}$ with $n\in \real$, define the set 
$$\mathcal J_I:=\{I'\in \dyadic: I\subset I'\subsetneq \widetilde I\}.$$
Then it follows from the triangle inequality that
$$|f_{I}-f_{\widetilde I}|\leq \sum_{I'\in \mathcal J_I} |f_{I'}-f_{\widehat I'}|.$$
Hence, by using Fubini's theorem, we have that
\begin{align*}
\sum_{I\in\dyadic_{\alpha(n)}} \nu(I) |f_{I}-f_{\widetilde I}|&\leq  \sum_{I\in\dyadic_{\alpha(n)}} \nu(I) \sum_{I'\in \mathcal J_I} |f_{I'}-f_{\widehat I'}|\\
& =\sum_{m=\alpha(n-1)+1}^{\alpha(n)} \sum_{I'\in \dyadic_m} |f_{I'}-f_{\widehat I'}| \left(\sum_{I\in \dyadic_{\alpha(n)}}\sum_{I'\in \mathcal J_I} \nu(I)\right).
\end{align*}
Notice that for any $I\in \dyadic_{\alpha(n)}$, we have $\nu(I)\approx e^{-\epsilon \alpha(n)Q} =K^{-\alpha(n)}$ and that for any $I'\in\dyadic_m$, the number of the dyadic elements $I\in \dyadic_{\alpha(n)}$ with $I'\in \mathcal J_I$ is $K^{\alpha(n)-m}$. Therefore, 
$$\sum_{I\in \dyadic_{\alpha(n)}}\sum_{I'\in \mathcal J_I} \nu(I) \approx K^{\alpha(n)-m-\alpha(n)}=K^{-m} = e^{-\epsilon \alpha(n)Q} \approx \nu(I'). $$
Hence, we have the estimate
 $$\sum_{I\in\dyadic_{\alpha(n)}} \nu(I) |f_{I}-f_{\widetilde I}|\lesssim \sum_{m=\alpha(n-1)+1}^{\alpha(n)} \sum_{I'\in \dyadic_m}\nu(I') |f_{I'}-f_{\widehat I'}|,$$
 and therefore the estimate
\begin{align*}
\|f\|_{\dlbesovo}&=\sum_{n=1}^{+\infty} \alpha(n)^\lambda\sum_{I\in\dyadic_{\alpha(n)}} \nu(I) |f_{I}-f_{\widetilde I}|\\
& \lesssim \sum_{n=1}^{+\infty} \alpha(n)^\lambda \sum_{m=\alpha(n-1)+1}^{\alpha(n)} \sum_{I'\in \dyadic_m} \nu(I') |f_{I'}-f_{\widehat I'}| \\
&\lesssim \sum_{m=1}^{+\infty} m^\lambda \sum_{I'\in \dyadic_m} \nu(I') |f_{I'}-f_{\widehat I'}| = \|f\|_{\dot{\mathcal B}^{0, \lambda}_1(\bx)}.
\end{align*}
Here in the last inequality, we used the fact that $m^\lambda>\alpha(n-1)^\lambda \geq \alpha(n)^\lambda/c_1^\lambda$ whenever $m>\alpha(n-1)$, where the constant $c_1$ is from the condition \eqref{sequence}. 
\end{proof}

\begin{example}\label{exam-strict}\rm
Let $X$ be a $2$-regular tree. We may identify each vertex of $X$ with a finite sequence formed by $0$ and $1$. For example, the children of the root can be denoted by $00$ and $01$. The children of the vertex $x=0\tau_1\cdots \tau_k$ is $0\tau_1\cdots\tau_k 0$ and $0\tau_1\cdots \tau_k 1$, where $\tau_i\in \{0, 1\}$. Moreover, each element $\xi$ of the boundary $\bx$ can be identified with an infinite sequence formed by $0$ and $1$. We denote $\xi=0\tau_1\tau_2\cdots$with $\tau_i\in \{0, 1\}$ when the geodesic from $0$ to $\xi$ passes through all the vertices $x_k=0\tau_1\cdots \tau_k$, $k\in \real$. 

We define a function $f$ on $\bx$ as follows: 
for $\xi=0\tau_1\tau_2\cdots \in \bx$ where $\tau_i\in \{0, 1\}$, we define
$$f(\xi)=\sum_{i=1}^{+\infty}  \frac{(-1)^{\tau_i}}{i^{\lambda+1}}.$$
Since the sum of $1/i^{\lambda+1}$ converges for $\lambda>0$, $f$ is well defined for all $\xi\in \bx$ and is bounded. Moreover, for any vertex $x=0\tau_1\cdots \tau_k$, it follows from the definition of $f$ that
\begin{equation}\label{exam-1}
f_{I_x}=\dashint_{I_x} f(\zeta)\, d\nu(\zeta) = \sum_{i=1}^{k} \frac{(-1)^{\tau_i}}{ i^{\lambda+1}}.
\end{equation}
Therefore, for the vertex $x$ above, we have
$$|f_{I_x}-f_{\widehat I_x}| = \frac{1}{ k^{\lambda+1}}.$$
Hence the $\dot{\mathcal B}^{0, \lambda}_1$-energy of $f$ is
\begin{align*}
\|f\|_{\dot{\mathcal B}^{0, \lambda}_1(\bx)}&=\sum_{n=1}^{+\infty} n^\lambda \sum_{I\in \dyadic_n} \nu(I) |f_{I}-f_{\widehat I}|\\
&= \sum_{n=1}^{+\infty} n^\lambda \sum_{I\in\dyadic_n} \nu(I) \frac{1}{ n^{\lambda+1}} =\sum_{n=1}^{+\infty} \frac{1}{n} =+\infty.
\end{align*}

On the other hand, for any $I\in\dyadic_{\alpha(n)}$, we have
\begin{equation}\label{dyadic_est}
|f_{I}-f_{\widetilde I}|=\left|\sum_{i=\alpha(n-1)+1}^{\alpha(n)} \frac{(-1)^{\tau_i}}{i^{\lambda+1}}\right|,
\end{equation}
where $\tau_i\in \{0, 1\}$ depends on $I$. We define a random series $\mathcal X_{\alpha(n)}$  by setting
$$\mathcal X_{\alpha(n)} =\sum_{i=\alpha(n-1)+1}^{\alpha(n)} \frac{{\sigma_i}}{i^{\lambda+1}},$$
where $(\sigma_i)_i$ are independent random variables with common distribution $P(\sigma_i=1)=P(\sigma_i=-1)=1/2$. Since the measure $\nu$ is a probability measure which is uniformly distributed on $\bx$, it follows from \eqref{dyadic_est} that
$$\sum_{I\in\dyadic_{\alpha(n)}} \nu(I)|f_I-f_{\widetilde I}| = \mathbb E(|\mathcal X_{\alpha(n)}|).$$
Here $\mathbb E(|\mathcal X_{\alpha(n)}|)$ is the expected value of $|\mathcal X_{\alpha(n)}|$. By the Cauchy-Schwarz inequality, $\mathbb E(|\mathcal X_{\alpha(n)}|)\leq (\mathbb E(\mathcal X_{\alpha(n)}^2) )^{1/2}$, we have that
\begin{align*}
\sum_{I\in\dyadic_{\alpha(n)}} \nu(I)|f_I-f_{\widetilde I}|&\leq (\mathbb E(\mathcal X_{\alpha(n)}^2) )^{1/2}=\left(\sum_{i, j=\alpha(n-1)+1}^{\alpha(n)} \frac{\mathbb E(\sigma_i\sigma_j)}{i^{\lambda+1}j^{\lambda+1}} \right)^{1/2}\\
&=\left(\sum_{i=\alpha(n-1)+1}^{\alpha(n)} \frac{\mathbb E({\sigma_i}^2)}{i^{2\lambda+2}} \right)^{1/2}=\left(\sum_{i=\alpha(n-1)+1}^{\alpha(n)} \frac{1}{i^{2\lambda+2}} \right)^{1/2}.
\end{align*}  
Here the second to last equality holds since $\sigma_i$ and $\sigma_j$ are independent for $i\not=j$ and $\mathbb E(\sigma_i\sigma_j)=\mathbb E(\sigma_i)\mathbb E(\sigma_j)=0$ for $i\not=j$.
Define $\alpha(n)=2^n$. Then we obtain that
$$\sum_{I\in\dyadic_{\alpha(n)}} \nu(I)|f_I-f_{\widetilde I}|\leq \left(\sum_{i=2^{n-1}+1}^{2^n} \frac{1}{i^{2\lambda+2}} \right)^{1/2} \leq \left(\sum_{i=2^{n-1}+1}^{2^n} \frac{1}{2^{(n-1)(2\lambda+2)}} \right)^{1/2}=\frac{1}{2^{(n-1)(\lambda+1/2)}}.$$
Therefore the $\dlbesovo$-energy of $f$ is estimated by
\begin{align*}
\|f\|_{\dlbesovo}&=\sum_{n=1}^{+\infty} \alpha(n)^\lambda \sum_{I\in\dyadic_{\alpha(n)}} \nu(I) |f_{I}-f_{\widetilde I}|\\
&\leq \sum_{n=1}^{+\infty} 2^{n\lambda}\frac{1}{2^{(n-1)(\lambda+1/2)}}=\sum_{n=0}^{+\infty} \frac{2^\lambda}{2^{n/2}} <+\infty.
\end{align*}

Hence $f\in \lbesovo$ while $f\notin {\mathcal B}^{0, \lambda}_1(\bx)$, and it follows that ${\mathcal B}^{0, \lambda}_1(\bx)$ is a strict subset of $\lbesovo$.
\end{example}

\subsection{Proof of Theorem \ref{th4}}
\begin{proof}
Let $p=(\beta-\log K)/\epsilon$ and $\lambda>p-1$ if $p>1$ or $\lambda\geq 0$ if $p=1$. From Proposition \ref{prop3.1}, the trace operator $T: N^{1,p}(X, \mul)\rarrow L^p(\bx)$ in Theorem \ref{th3}  is bounded and linear.  Now we define an extension operator $E$ by using \eqref{extension-operator1}, \eqref{extension-operator2} and \eqref{extension-operator3}. It is easy to see that the extension  $Eu$ is well defined for any function $u\in L^1_{\rm loc}(\bx)$ and that $T\circ E$ is the identity operator on $L^1_{\rm loc}(\bx)$.

Repeating the estimates in { Extension Part} of the proof of Theorem \ref{th2}, for $\theta=1-(\beta-\log K)/(p\epsilon)=0$, we also have the following estimates:
\begin{equation}\label{optimal1}
\int_X |g_{\tilde u}|^p\, d\mul\approx \|u\|^p_{{\dot {\mathcal B}}_p^{0, \lambda}(\bx)}
\end{equation}
and
\begin{equation}\label{optimal2}
\int_X |\tilde u|^p\, d\mu\lesssim \int_\bx |u|^p\, d\nu.
\end{equation}
Hence the extension operator $E$ is bounded and linear from ${ {\mathcal B}}_p^{0, \lambda}(\bx)$ to $N^{1,p}(X, \mul)$.

Moreover, since $u$ is the trace of $\tilde u$, by Theorem \ref{th3} and  Proposition \ref{prop3.1}, we have 
$$\|u\|_{L^p(\bx)}\lesssim \|\tilde u\|_{N^{1,p}(X, \mul)}.$$
Combining the above inequality with \eqref{optimal1} and \eqref{optimal2}, we obtain the estimate
\begin{equation}\label{f-optimal}
\| u\|_{{ {\mathcal B}}_p^{0, \lambda}(\bx)}\approx \|\tilde u\|_{N^{1,p}(X, \mul)}.
\end{equation} 
Hence the ${{ {\mathcal B}}_p^{0, \lambda}(\bx)}$-norm of $u$ is comparable to the ${N^{1,p}(X, \mul)}$-norm of $\tilde u= Eu$. Thus    ${{ {\mathcal B}}_p^{0, \lambda}(\bx)}$ is the optimal space for which $E$ is both bounded and linear.

\end{proof}

\noindent Pekka Koskela

\noindent
Department of Mathematics and Statistics, University of Jyv\"askyl\"a, PO~Box~35, FI-40014 Jyv\"askyl\"a, Finland

\noindent{\it E-mail address}:  \texttt{pekka.j.koskela@jyu.fi}

\medskip
\noindent Zhuang Wang

\noindent
Department of Mathematics and Statistics, University of Jyv\"askyl\"a, PO~Box~35, FI-40014 Jyv\"askyl\"a, Finland

\noindent{\it E-mail address}:  \texttt{zhuang.z.wang@jyu.fi}


\begin{thebibliography}{aaaaa}

\bibitem{Ar} N.~Aronszajn: {\it Boundary values of functions with finite Dirichlet integral}, Techn.~Report 14, University of Kansas, 1955.


\bibitem{BB11} A.~Bj\"orn and J.~Bj\"orn: {\it Nonlinear potential theory on metric spaces}, EMS Tracts Math. 17, European Mathematical Society, Zurich 2011.

\bibitem{BBGS} A.~Bj\"orn, J.~Bj\"orn, J. T.~Gill and N.~Shanmugalingam: {\it Geometric analysis on Cantor sets and trees.} J. Reine Angew. Math. 725 (2017), 63-114.


\bibitem{BBS03} A.~Bj\"orn, J.~Bj\"orn and N.~Shanmugalingam: {\it The Dirichlet problem for p-harmonic functions on metric spaces}, J. Reine  Angew. Math. 556 (2003), 173-203.

\bibitem{BHK01} M.~Bonk, J.~Heinonen and P.~Koskela: {\it Uniformizing Gromov hyperbolic spaces}, Astérisque No. 270 (2001), viii+99 pp.

\bibitem{BS18} M.~Bonk and E.~Saksman: {\it Sobolev spaces and hyperbolic fillings}, J. Reine Angew. Math. 737 (2018), 161-187.

\bibitem{BH99} M.~Bridson and  A.~Haefliger: {\it Metric spaces of non-positive curvature,} Grundlehren Math. Wiss. 319, Springer-Verlag, Berlin 1999. 

\bibitem{BG79} V. I. Burenkov and  M. L. Goldman: {\it Extension of functions from $L_p$}, Studies in the theory of differentiable functions of several variables and its applications, VII, 
Trudy Mat. Inst. Steklov. 150 (1979), 31-51, 321.


\bibitem{FJS} W. Farkas, J. Johnsen and W. Sickel: {\it Traces of anisotropic Besov-Lizorkin-Triebel spaces-a complete treatment of the borderline cases}, Math. Bohem. 125 (2000), no. 1, 1-37. 




\bibitem{Ga} E.~Gagliardo: {\it Caratterizzazioni delle tracce sulla frontiera relative ad alcune classi di funzioni in $n$ variabili}, Rend.~Sem.~Mat.~Univ.~Padova 27 (1957), 284--305.

\bibitem{Gin84} A. Ginzburg, {\it Traces of functions from weighted classes}, Izv. Vyssh. Uchebn. Zaved. Mat. (1984) 61–64.

\bibitem{H03} P.~Haj\l asz: {\it Sobolev space on metric-measure spaces, in Heat kernels and analysis on manifolds, graphs and metric spaces (Paris 2002),} Contemp. Math. 338, American Mathematical Society, Providence (2003), 173-218.

\bibitem{HP} P.~Haj\l asz and P. Koskela: {\it Sobolev met Poincar\'e}, Mem. Amer. Math. Soc. (2000), no. 688, x+101 pp.

\bibitem{HS10} D. Haroske and H. J. Schmeisser: {\it On trace spaces of function spaces with a radial weight: the atomic approach}, Complex Var. Elliptic Equ. 55 (2010), no. 8-10, 875-896. 

\bibitem{H01} J.~Heinonen: {\it Lectures on analysis on metric spaces}, Universitext, Springer-Verlag, New York 2001.

\bibitem{HK98} J.~Heinonen and P.~Koskela: {\it Quasiconformal mappings in metric spaces with controlled geometry,} Acta Math. 181 (1998), 1-61.

\bibitem{HKST} J.~Heinonen, P.~Koskela, N.~Shanmugalingam and J.~Tyson: {\it Sobolev Spaces on Metric Measure Spaces: An Approach Based on Upper Gradients}. Cambridge: Cambridge University Press, 2015.

\bibitem{J00} J. Johnsen: {\it Traces of Besov spaces revisited}, Z. Anal. Anwendungen 19 (2000), no. 3, 763-779. 

\bibitem{JoWa80} A.~Jonsson and H.~Wallin: {\it The trace to subsets of $\real^n$ of Besov spaces in the general case}, Anal. Math. 6 (1980), 223-254.


\bibitem{KKZ}
A. Kauranen, P. Koskela and A. Zapadinskaya:
{\it Regularity and Modulus of Continuity of Space-Filling Curves,} to appear in J. Analyse Math.



\bibitem{Ma} L.~Mal\'y: {\it Trace and extension theorems for Sobolev-type functions in metric spaces}, arXiv:1704.06344.

\bibitem{MSS} L.~Mal\'y, N.~Shanmugalingam and M.~Snipes: {\it Trace and extension theorems for functions of bounded variation}, Ann. Sc. Norm. Super. Pisa Cl. Sci. (5) 18 (2018), no. 1, 313-341. 


\bibitem{KTW} P.~Koskela, T.~Soto and Z.~Wang: {\it Traces of weighted function spaces: dyadic norms and Whitney extensions.} Sci. China Math. 60 (2017), no. 11, 1981-2010.



\bibitem{Pe} J.~Peetre: {\it A counterexample connected with Gagliardo's trace theorem}, Special issue dedicated to Władysław Orlicz on the occasion of his seventy-fifth birthday,
Comment. Math. Special Issue 2 (1979), 277-282.



\bibitem{SS} E.~Saksman and T.~Soto: {\it Traces of Besov, Triebel-Lizorkin and Sobolev spaces on metric spaces}, Anal. Geom. Metr. Spaces 5 (2017), 98-115. 



\bibitem{So} T.~Soto: {\it  Besov spaces on metric spaces via hyperbolic fillings}, arXiv:1606.08082.


\bibitem{T83} H.~Triebel: {\it Theory of function spaces}, Monographs in Mathematics, 78. Birkh\"auser Verlag, Basel, 1983.

\bibitem{T01} H.~Triebel: {\it The structure of functions}, Monographs in Mathematics, 97. Birkh\"auser Verlag, Basel, 2001.


\bibitem{Tyu1} A.~I.~Tyulenev: {\it Description of traces of functions in the Sobolev space with a Muckenhoupt weight}, Proc.~Steklov Inst.~Math.~284 (2014), no. 1, 280--295.

\bibitem{Tyu2} A.~I.~Tyulenev: {\it Traces of weighted Sobolev spaces with Muckenhoupt weight. The case $p = 1$}, Nonlinear Anal. 128 (2015), 248--272.



\end{thebibliography}
\end{document}